\documentclass[11pt]{amsart}

\usepackage[square,compress,comma, numbers,sort]{natbib}
\usepackage[colorlinks=true, citecolor=red, linkcolor=blue]{hyperref}
\usepackage{amsfonts,mathtools}

\allowdisplaybreaks[4]

\usepackage{amsmath}
\usepackage{bbm}
\usepackage{amssymb}
\usepackage{mathtools}

\usepackage{amsmath, amssymb, amsfonts, amsthm}

\newcommand{\e}{\mathbb{E}}
\newcommand{\E}[1]{\mathbb{E}\left\{ #1\right\}}

\newcommand{\pk}[1]{\mathbb{P} \left\{ #1 \right \} }

\newcommand{\R}{\mathbb{R}}
\newcommand{\li}{\text{Li}_{\frac{1}{2}}}
\newcommand{\Z}{\mathbb{Z}}
\newcommand{\N}{\mathbb{N}}

\newcommand{\ep}{\varepsilon}

\newcommand{\eqd}{\overset{d}{=}}

\newcommand{\limit}[1]{\lim_{#1 \to   \infty}}

\newcommand{\ind}{\mathbb{I}}

\def\IF{\infty}
\def\bqny#1{\begin{eqnarray*} #1 \end{eqnarray*}}
\def\bqn#1{\begin{eqnarray} #1 \end{eqnarray}}

\newcommand{\kb}[1]{\boldsymbol{#1}}
\newcommand{\vk}[1]{\kb{#1}}

\DeclareMathOperator{\cov}{\mathbb{C}ov}
\DeclareMathOperator{\var}{\mathbb{V}ar}

\newtheorem{theorem}{Theorem}[section]
\newtheorem{proposition}[theorem]{Proposition}
\newtheorem{lemma}[theorem]{Lemma}
\newtheorem{corollary}[theorem]{Corollary}
\newtheorem{remark}[theorem]{Remark}

\begin{document}

\title{On the speed of convergence of discrete Pickands constants to continuous ones}

\author{Krzysztof Bisewski}
\address{Krzysztof Bisewski, Department of Actuarial Science, University of Lausanne,	UNIL-Dorigny, 1015 Lausanne, Switzerland}
\email{Krzysztof.Bisewski@unil.ch}

\author{Grigori Jasnovidov*}
\address{Grigori Jasnovidov*, Laboratory of Statistical Methods, St. Petersburg Department of Steklov Mathematical Institute of Russian Academy of Sciences,
}
\email{griga1995@yandex.ru*}

\bigskip

\date{\today}
 \maketitle

 {\bf Abstract:} In this manuscript, we address open questions raised by Dieker \& Yakir (2014), who proposed a novel method of estimation of (discrete) Pickands constants $\mathcal{H}^\delta_\alpha$ using a family of estimators $\xi^\delta_\alpha(T),T>0$, where $\alpha\in(0,2]$ is the Hurst parameter, and $\delta\geq0$ is the step-size of the regular discretization grid. We derive an upper bound for the discretization error $\mathcal{H}_\alpha^0 - \mathcal{H}_\alpha^\delta$, whose rate of convergence agrees with Conjecture 1 of Dieker \& Yakir (2014) in case $\alpha\in(0,1]$ and agrees up to logarithmic terms for $\alpha\in(1,2)$. Moreover, we show that all moments of $\xi_\alpha^\delta(T)$ are uniformly bounded and the bias of the estimator decays no slower than $\exp\{-\mathcal CT^{\alpha}\}$, as $T$ becomes large.\\


\section{Introduction}\label{s:introduction}
For any $\alpha\in(0,2]$ let $\{B_\alpha(t), t\in\R\}$ be a
fractional Brownian motion (later on, fBm) with Hurst parameter $H=\alpha/2$,
that is, $B_\alpha(t)$ is a centered Gaussian process with covariance function given by
\bqny{
\cov(B_\alpha(t),B_\alpha(s)) = \frac{|t|^{\alpha}+|s|^{\alpha}-|t-s|^{\alpha}}{2}, \quad t,s \in\R,~\alpha\in (0,2].
}
In this manuscript we consider the classical Pickands constant defined by
\bqn{\label{classical_pick_const_def}
\mathcal{H}_{\alpha} := \limit{S} \frac{1}{S}
\E{\sup_{t\in [0,S]}e^{\sqrt 2 B_\alpha(t)-t^{\alpha}}}\in (0,\infty), \quad \alpha\in (0,2].
}

The constant $\mathcal H_\alpha$ was first defined by Pickands
\cite{MR250367, MR250368} to describe the asymptotic behavior of the maximum of stationary Gaussian processes. Since then, Pickands constants played an important role in the theory of Gaussian processes, appearing in various asymptotic results related to the supremum; see monographs \citep{Pit96, 20lectures}.
In \cite{MR3351818}, it was recognized that discrete Pickands constant can be interpreted as an \textit{extremal index} of a Brown-Resnick process. This new realization motivated generalization of Pickands constants beyond the realm of Gaussian processes. We refer to \citep{SBK,MR3745388} for further references, who give an excellent account of the history of Pickands constants, their connection to the theory of max-stable processes and most recent advances in the theory.

While being omnipresent in the asymptotic theory of stochastic processes, to this date, the value of $\mathcal H_\alpha$ is known only in two very special cases: $\alpha=1$ and $\alpha=2$. In these cases,  the distribution of the supremum of process $B_\alpha$ is well-known --- $B_1$ is a standard Brownian motion, while $B_2$ is a straight line with random, normally distributed slope. When $\alpha\not\in\{1,2\}$, one may attempt to estimate the numerical value of $\mathcal H_\alpha$ from the definition \eqref{classical_pick_const_def} using Monte-Carlo methods. However, there is a number of problems associated with this approach:
\begin{enumerate}
\item[(i)] Firstly, Pickands constant $\mathcal H_\alpha$ in \eqref{classical_pick_const_def} is defined as a limit $S\to\infty$, so one must approximate it by choosing some (large) $S$. This results in bias in estimation, which we call the \textit{truncation error}. The truncation error was shown to decay faster than $S^{-p}$ for any $p<1$, \cite[Corollary~3.1]{MR2222683}.
\item[(ii)] Secondly, for every $\alpha\in(0,2)$, the variance of the truncated estimator blows up, as $S\to\infty$, i.e.
\[\lim_{S\to\infty}\var \left\{\frac{1}{S}
\sup_{t\in [0,S]}\exp\{\sqrt 2 B_\alpha(t)-t^{\alpha}\}\right\} = \infty.\]
This can be easily seen by considering the second moment of $\tfrac{1}{S}\exp\{\sqrt{2}B_\alpha(S)-S^\alpha\}$. This directly affects the \textit{sampling error} (standard deviation) of the Crude Monte Carlo estimator. As $S\to\infty$, one needs more and more samples to prevent its variance from blowing up.
\item[(iii)] Finally, there are no methods available for the exact simulation of $\sup_{t\in [0,S]}\exp\{\sqrt 2 B_\alpha(t)-t^{\alpha}\}$ for $\alpha\not\in\{1,2\}$. One must therefore resort to some method of approximation. Typically, one would simulate fBm on a regular $\delta$-grid, i.e. on the set $\delta\Z$ for $\delta>0$; cf. Eq.~\eqref{delta_pick_const_def} below. This approximation leads to a bias, which we call the \textit{discretization error}.
\end{enumerate}
In the following, for any fixed $\delta>0$ we define the discrete Pickands constant
\bqn{\label{delta_pick_const_def}
\mathcal{H}_{\alpha}^{\delta} := \limit{S} \frac{1}{S}
\E{\sup_{t\in [0,S]_\delta}e^{\sqrt 2 B_\alpha(t)-t^{\alpha}}
}, \quad \alpha\in(0,2],}
where for $a,b\in\R$ and $\delta>0$ $[a,b]_\delta = [a,b]\cap \delta\Z$. Additionally, we set $0\Z = \R$, so that $\mathcal H^0_\alpha = \mathcal H_\alpha$. In light of the discussion in item (iii) above, the discretization error equals to $\mathcal H_\alpha-\mathcal H_\alpha^\delta$. We should note that the quantity $\mathcal H_\alpha^\delta$ is well-defined and $\mathcal{H}_{\alpha}^{\delta}\in (0,\infty)$ for $\delta\ge0$. Moreover, $\mathcal{H}_{\alpha}^{\delta} \to \mathcal{H}_{\alpha}$, as $\delta\to0$,
which means that the discretization error diminishes,
as the size of the gap of the grid goes to $0$. We refer to \citep{SBK} for the proofs of these properties.

In recent years, \citep{MR3217455} proposed a new representation of $\mathcal H^\delta_\alpha$, which does not involve the limit operation. They show
\cite[Proposition~3]{MR3217455}, that for all $\delta\ge 0$, and $\alpha\in(0,2]$:
\bqn{\label{xi_alpha_delta}
\mathcal{H}_{\alpha}^\delta = \E{\xi_\alpha^\delta}, \quad \text{where} \quad \xi_\alpha^\delta := \frac{
\sup_{t\in\delta\Z}e^{\sqrt 2B_\alpha(t)-|t|^{\alpha}}}{
\delta\sum_{t\in \delta \Z}e^{\sqrt 2B_\alpha(t)-|t|^{\alpha} }},
}
where for $\delta = 0$ the denominator in the fraction above is substituted by $\int_{\R} e^{\sqrt 2B_\alpha(t)-|t|^{\alpha}}{\rm d}t$. In fact, the denominator can be substituted with $\eta\sum_{t\in \eta \Z}e^{\sqrt 2B_\alpha(t)-|t|^{\alpha}}$ for any $\eta$, which is an integer multiple of $\delta$, see \citep[Theorem~2]{SBK}. While one would ideally estimate $\mathcal H_\alpha$ using $\xi_\alpha^0$, it is unfortunately unfeasible due to the lack of exact simulation methods of $\xi^\delta_\alpha$ (see also item (iii) above). For that reason, the authors define the `truncated' version of random variable $\xi_\alpha^\delta$, namely
\begin{align*}
\xi_\alpha^\delta(T) := \frac{
\sup_{t\in[-T,T]_\delta}e^{\sqrt 2B_\alpha(t)-|t|^{\alpha}}}{
\delta\sum_{t\in[-T,T]_\delta}e^{\sqrt 2B_\alpha(t)-|t|^{\alpha} }},
\end{align*}
where for $\delta = 0$ the denominator in the fraction above is substituted by $\int_{-T}^T e^{\sqrt 2B_\alpha(t)-|t|^{\alpha}}{\rm d}t$. For any $\delta,T\in(0,\infty)$, the estimator $\xi_\alpha^\delta(T)$ is a functional of a fractional Brownian motion on a finite grid and, as such, it can be simulated exactly; see e.g. \cite{dieker2004simulation} for the survey of methods of simulation of fBm. The side effect of this approach is that the new estimator induces both the truncation and discretization errors described in items (i) and (iii) above.

In this manuscript we rigorously show that the estimator $\xi_\alpha^\delta(T)$ is well-suited for simulation. In Theorem~\ref{thm:main}, we address the conjecture about the asymptotic behavior of the discretization error between the continuous and discrete Pickands constant for a fixed $\alpha\in(0,2]$, which was stated by the inventors of the estimator $\xi^\delta_\alpha$, namely:
\begin{quote}
{\bf \cite[Conjecture~1]{MR3217455}} \quad For all $\alpha\in(0,2]$ it holds, that $\displaystyle \lim_{\delta\to 0}
\frac{\mathcal{H}_{\alpha}-\mathcal{H}_{\alpha}^{\delta}}{\delta^{\alpha/2}}
\in (0,\infty).$
\end{quote}
We establish that the conjecture is \emph{true} when $\alpha=1$ and it is \emph{not true} when $\alpha=2$; see Theorem~\ref{thm:main}(iii-iv) below, where the exact asymptotics of the discretization error are derived in these two special cases. Furthermore, in Theorem~\ref{thm:main}(i), we show that $\limsup_{\delta\to0}\delta^{-\alpha/2}(\mathcal{H}_{\alpha}-\mathcal{H}_{\alpha}^{\delta}) \in(0,\infty)$ for $\alpha\in(0,1)$ and in Theorem~\ref{thm:main}(ii) we show that $\mathcal H_\alpha-\mathcal H_\alpha^\delta$ is upper-bounded by $\delta^{\alpha/2}$ up to logarithmic terms for $\alpha\in(1,2)$ and all $\delta>0$ small enough. These results support the claim of the conjecture for all $\alpha\in(0,2)$.\\
Secondly, we consider the truncation and sampling errors induced by $\xi_\alpha^\delta(T)$. In Theorem~\ref{thm:tail_behavior_pickands} we derive a uniform upper bound for the tail of $\xi_\alpha^\delta$ which implies that all moments of $\xi_\alpha^\delta$ exist and are
uniformly bounded in $\delta\in[0,1]$ . In Theorem~\ref{theorem_truncation} we
establish that for any $\alpha\in(0,2)$ and $p\ge1$, the difference $|\e(\xi_\alpha^\delta(T))^p - \e(\xi_\alpha^\delta)^p|$ decays no slower than $\exp\{-\mathcal CT^\alpha\}$, as $T\to\infty$, uniformly for all $\delta\in[0,1]$. This implies that the truncation error of the Dieker-Yakir estimator decays no slower than $\exp\{-\mathcal CT^\alpha\}$ and together with Theorem~\ref{thm:tail_behavior_pickands} they imply that $\xi_\alpha^\delta(T)$ has a uniformly bounded sampling error, i.e.
\begin{equation}\label{eq:pth_moment_uniform_boundedness}
\sup_{(\delta,T)\in[0,1]\times[1,\infty)} \var\left\{ \xi_\alpha^\delta(T)\right\} < \infty.
\end{equation}

The manuscript is organized as follows. In Section \ref{s:main_results} we present our main results and discuss their extensions and relation to other problems. All the proofs are provided in Section \ref{s:proofs}.

\section{Main Results}\label{s:main_results}

In the following, we give an upper bound for
$\mathcal H_{\alpha} - \mathcal H_{\alpha}^{\delta}$ for all $\alpha\in(0,1)\cup(1,2)$ for small $\delta>0$. When $\alpha\in\{1,2\}$ we provide the exact asymptotics of the discretization error, as $\delta\to0$. Below, $\zeta$ is the
Euler-Riemann zeta function.
\begin{theorem}\label{thm:main}
It holds, that
\begin{itemize}
\item[(i)] for all $\alpha\in(0,1)$ there exists $\mathcal C>0$ such that for all $\delta>0$ sufficiently small,
\begin{equation*}
\mathcal H_{\alpha} - \mathcal H_{\alpha}^{\delta} \leq \mathcal C\delta^{\alpha/2};
\end{equation*}
\item[(ii)] for all $\alpha\in(1,2)$
there exists $\mathcal C>0$ such that for all
 $\delta>0$ sufficiently small,
\begin{equation*}
\mathcal H_{\alpha} - \mathcal H_{\alpha}^{\delta} \leq
\mathcal C\delta^{\alpha/2}|\log\delta|^{1/2};
\end{equation*}
\item[(iii)] $\displaystyle \lim_{\delta\to0}\frac{\mathcal H_1-\mathcal H_1^\delta}{\sqrt{\delta}} = -\frac{\zeta(1/2)}{\sqrt \pi}$;
\item[(iv)] $\displaystyle \lim_{\delta\to0}\frac{\mathcal H_2-\mathcal H_2^\delta}{\delta^2} = \frac{1}{12\sqrt{\pi}}$.
\end{itemize}
\end{theorem}

The exact asymptotic behavior of the discretization error, as $\delta\to0$ could be derived for $\alpha\in\{1,2\}$ because, in these cases, the explicit formulas for $\mathcal H_\alpha^\delta$
are known. For convenience, we collect these results in Proposition~\ref{prop:formulas_alpha12} below. In the following, $\Phi$ is the cumulative distribution function of a standard Gaussian random variable.
\begin{proposition}\label{prop:formulas_alpha12}
It holds, that
\begin{itemize}
\item[(i)] $\mathcal H_1 = 1$, \ and \ $\displaystyle \mathcal{H}^{\delta}_1 =
\bigg(\delta\exp\Big\{2\sum_{k=1}^\infty\frac{\Phi(-\sqrt {\delta k/2})}{k}\Big\}\bigg)^{-1}$ for all $\delta>0$;
\item[(ii)] $\mathcal{H}_{2}=\frac{1}{\sqrt{\pi}}$, \ and \ $\displaystyle \mathcal{H}^\delta_{2}=\frac{2}{\delta}\left( \Phi(\delta/\sqrt{2})-\frac{1}{2}   \right)$ for all $\delta>0$.
\end{itemize}
\end{proposition}

In the following two results we establish an upper bound for the complementary cdf of $\xi_\alpha^\delta$ and for the truncation error discussed in item (i) in Section \ref{s:introduction}. These two results combined imply that the sampling error of $\xi_\alpha^\delta(T)$ is uniformly bounded in $(\delta,T)\in[0,1]\times[1,\infty)$, cf. Eq.~\eqref{eq:pth_moment_uniform_boundedness}.

\begin{theorem}\label{thm:tail_behavior_pickands}
For any $\alpha \in (0,2)$ there exist positive constants
$\mathcal C_1,\mathcal C_2$ such that
\bqny{
\pk{\xi_\alpha^\delta>x}\le \mathcal C_1e^{-\mathcal C_2\log ^2 x}
}
for all $x\ge 1,\delta\in[0,1]$.
\end{theorem}

Evidently, Theorem~\ref{thm:tail_behavior_pickands} implies that all moments of
$\xi_\alpha^\delta$ are finite
and uniformly bounded in $\delta\in [0,1]$ for any fixed $\alpha \in (0,2)$.

\begin{theorem}\label{theorem_truncation}
For any $\alpha\in (0,2)$ and $p>0$ there exist postive constants $\mathcal C_1,\mathcal C_2$ such that
\bqny{
\left|\E{(\xi_\alpha^\delta(T))^p} -
\E{(\xi_\alpha^\delta)^p} \right| \le \mathcal C_1e^{-\mathcal C_2T^\alpha}
}
for all $(\delta,T)\in[0,1]\times[1,\infty)$.
\end{theorem}

\subsection{Discussion}

We believe that finding the exact asymptotics of the speed of the discretization error $\mathcal H_{\alpha}-\mathcal H_\alpha^\delta$ is closely related to the behavior of fractional Brownian motion around the time of its supremum. We motivate this by the following heuristic:
\begin{align*}
\mathcal H_\alpha - \mathcal H_\alpha^\delta &= \E{\frac{
\sup_{t\in\R}e^{\sqrt 2B_\alpha(t)-|t|^{\alpha}} - \sup_{t\in\delta\Z}e^{\sqrt 2B_\alpha(t)-|t|^{\alpha}}}{
\delta\sum_{t\in \delta \Z}e^{\sqrt 2B_\alpha(t)-|t|^{\alpha} }}} \\
& \approx \E{\Delta(\delta) \cdot \frac{
\sup_{t\in\delta\Z}e^{\sqrt 2B_\alpha(t)-|t|^{\alpha}}}{
\delta\sum_{t\in \delta \Z}e^{\sqrt 2B_\alpha(t)-|t|^{\alpha} }}}, \\
& \approx \E{\Delta(\delta)}\cdot \mathcal H^\delta_\alpha,
\end{align*}
where $\Delta(\delta)$ is the difference between the supremum on the continuous and discrete grid, i.e. $\Delta(\delta) := \sup_{t\in\R}\{\sqrt{2}B_\alpha(t)-|t|^{\alpha}\} - \sup_{t\in\delta\Z}\{\sqrt{2}B_\alpha(t)-|t|^{\alpha}\}$. The first approximation above is due to the mean value theorem and the second approximation is based on the assumption that $\Delta(\delta)$ and $\xi^\delta_\alpha$ are asymptotically independent, as $\delta\to0$. Now, we believe that $\Delta(\delta)\sim \mathcal C\delta^{\alpha/2}$ due to self-similarity, where $\mathcal C>0$ is some constant, which would imply that $\mathcal H_\alpha-\mathcal H_\alpha^\delta \sim \mathcal C\mathcal H_\alpha\delta^{\alpha/2}$. This heuristic reasoning can be made rigorous in case $\alpha=1$, when $\sqrt{2}B_\alpha(t)-|t|^\alpha$ is a L{\'e}vy process (Brownian motion with drift). In this case, the asymptotic behavior of functionals such as $\E{\Delta(\delta)}$, as $\delta\to0$ can be explained by the weak convergence of trajectories around the time of supremum to the so-called \emph{L{\'e}vy process conditioned to be positive}, see \cite{ivanovs2018zooming} for more information on the topic. In fact, Theorem~\ref{thm:main}(iii) can be proven using the tools developed in \cite{bisewski2020zooming}. To the best of the authors' knowledge, there are no such results available for a general fractional Brownian motion. Although, it is worth mentioning that recently \citep{MR4158797} considered a related problem of penalizing fractional Brownian motion for being negative.

A problem related to the asymptotic behavior of $\mathcal H_\alpha- \mathcal H_\alpha^\delta$ was considered in \citep{MR3574693, MR3776210}, who have shown that $\e \sup_{t\in [0,1]}B_\alpha(t)-\e \sup_{t\in [0,1]_\delta}B_\alpha(t)$, decays like $\delta^{\alpha/2}$ up to logarithmic terms. We should emphasize that in Theorem~\ref{thm:main}, case $\alpha\in(0,1)$ we were able to establish that the upper bound for the discretization error decays \emph{exactly} like $\delta^{\alpha/2}$. In light of the discussion above, we believe that the result and the proving methodology of Theorem~\ref{thm:main}(i), could be useful in further research related to the discretization error for fractional Brownian motion.\\

\textit{Discretization error for asymptotic constants in the ruin theory for Gaussian processes.} Although arguably most celebrated, Pickands constants are not the only constants appearing in the asymptotic theory of Gaussian processes. Depending on the setting, other constants might appear. Among others, we distiguish: Parisian Pickands constants \cite{thirdprojectParisian, ParisRuinGenrealHashorvaJiDebicki, ParisianFiniteHoryzon}, sojourn Pickands constants \cite{SojournInfty,DebZbiXia}, Piterbarg-type constants \cite{20lectures,Pit96,Lanpeng2BM,longKrzys}, and generalized Pickands constants \cite{DE2002,MR2111193}. Just like the classical Pickands constants, their numerical values are known typically only in case $\alpha\in\{1,2\}$. In other cases, they need to be estimated and one encounters problems described in items (i-iii) in Section \ref{s:introduction}. In particular,
possibility is an approximation by discretization. We believe,
that under appropriate assumptions, using the technique from the proof of Theorem~\ref{thm:main}(ii), one could derive upper bounds for the discretization error, which are exact up to logarithmic terms.\\

\textit{Monotonicity of Pickands constants.} Based on the definition \eqref{delta_pick_const_def}, it is clear that for any $\alpha\in(0,2)$, the sequence $\{\mathcal H^{\delta}_\alpha, \mathcal H^{2\delta}_\alpha,\mathcal H^{4\delta}_\alpha, \ldots\}$ is decreasing for any fixed $\delta>0$. It is therefore natural to speculate
that $\delta\mapsto\mathcal H_\alpha^\delta$ is a decreasing function.
The explicit formulas for $\mathcal H_1^\delta$ and $\mathcal H_2^\delta$ given in Proposition~\ref{prop:formulas_alpha12} allow us to give the positive answer to this question in these cases. Namely,
\begin{corollary} \label{corollary_decreasing_Pickands_constant}
$\mathcal{H}^{\delta}_1$ and $\mathcal{H}^{\delta}_2$ are strictly
decreasing functions with respect to $\delta$ for all $\delta\ge 0$.
\end{corollary}

\section{Proofs}\label{s:proofs}
In this section we give proofs. Define for $\alpha \in (0,2)$
$$Z_\alpha(t) = \sqrt 2 B_{\alpha}(t)-|t|^{\alpha},
\ \  \ t \in\R .$$
Assume that all considered random processes and variables are
defined on a complete general probability space $\Omega$ equipped with
a probability measure $\mathbb P$.
Let $\mathcal C,\mathcal C_1,\mathcal C_2,\ldots$ be some positive constants that may differ from line to line.

\subsection{Proof of Theorem~\ref{thm:main} case $\alpha\in(0,1)$}

The proof of Theorem~\ref{thm:main} in case $\alpha\in(0,1)$
is based on the following three results. In what follows, $\eta$ is independent of $\{Z_\alpha(t), t\in\R\}$ and follows a standard exponential distribution.

\begin{lemma}\label{lem:pickands_diff}
For all $\alpha\in (0,2)$
\begin{align*}
\mathcal H^{\delta/2}_{\alpha} - \mathcal H^{\delta}_{\alpha} & = \delta^{-1}\pk{\sup_{t\in\delta\Z\setminus\{0\}} Z_\alpha(t) < 0, \ \sup_{t\in\delta\Z\setminus\{0\}} Z_\alpha(t - \tfrac{\delta}{2}\cdot{\rm sgn}(t)) + \eta < 0}.
\end{align*}
\end{lemma}

As a side note, we remark that the representation in Lemma~\ref{lem:pickands_diff} yields a straightforward lower bound $\mathcal H^{\delta/2}_{\alpha} - \mathcal H^{\delta}_{\alpha} \geq \delta^{-1}\pk{\sup_{t\in(\delta/2)\Z\setminus\{0\}} Z_\alpha(t) + \eta < 0}$ for all $\alpha\in(0,2)$, $\delta>0$.

\begin{lemma}\label{lem:pickands_diff_bounds}
For all $\alpha\in(0,1)$ and $\delta>0$
\begin{equation*}
\mathcal H^{\delta/2}_{\alpha} - \mathcal H^{\delta}_{\alpha} \leq \delta^{-1}\pk{\sup_{t\in \delta\Z\setminus\{0\}} Z_\alpha(t) + \eta < 0}.
\end{equation*}
\end{lemma}

\begin{proposition}\label{prop:DY_plus_eta} For any $\alpha\in(0,2)$, there exists $\mathcal C> 0$ such that
\begin{align*}
\pk{\sup_{t\in\delta\Z\setminus\{0\}} Z_\alpha(t) + \eta < 0} \leq \mathcal C \delta^{1+\alpha/2}
\end{align*}
for all $\delta>0$ small enough.
\end{proposition}

\begin{proof}[Proof of Theorem~\ref{thm:main}, $\alpha\in (0,1)$.]
Using the fact that $\mathcal H^\delta_\alpha \to \mathcal H_\alpha$, as $\delta\downarrow0$, we may represent the discretization error $\mathcal H_\alpha - \mathcal H^\delta_\alpha$, as a telescopic series, that is
\begin{equation*}
\mathcal H_\alpha - \mathcal H^\delta_\alpha = \sum_{k=0}^\infty \mathcal H^{2^{-(k+1)}\delta}_\alpha - \mathcal H^{2^{-k}\delta}_\alpha.
\end{equation*}
Combining Lemma~\ref{lem:pickands_diff_bounds} and Proposition~\ref{prop:DY_plus_eta} we find that there exists $\mathcal C_0>0$ such that
\begin{align*}
\mathcal H_\alpha - \mathcal H^\delta_\alpha \leq \sum_{k=0}^\infty \mathcal C_0 (2^{-k}\delta)^{\alpha/2}
\end{align*}
for all $\delta$ small enough. This completes the proof with
$\mathcal C = \frac{\mathcal C_0}{1-2^{-\alpha/2}}$.
\end{proof}

We remark that if the upper bound in Lemma~\ref{lem:pickands_diff_bounds} holds also for $\alpha\in(1,2)$, then the upper bound in Theorem~\ref{thm:main}(i) would hold for all $\alpha\in(0,2)$.
The remainder of this section is devoted to proving Lemma~\ref{lem:pickands_diff}, Lemma~\ref{lem:pickands_diff_bounds}, and Proposition~\ref{prop:DY_plus_eta}.

In what follows, for any $\alpha\in(0,2)$, let $\{X_\alpha(t), t\in\R\}$ be a centered, stationary Gaussian process with $\var\{X_\alpha(t)\} = 1$, whose covariance function satisfies
\begin{equation}\label{eq:X_alpha_cov_assumption}
\cov(X_\alpha(t),X_\alpha(0)) = 1-|t|^\alpha+o(|t|^\alpha), \quad  t \to 0.
\end{equation}
Before we give the proof of Lemma~\ref{lem:pickands_diff}, we introduce the following result.

\begin{lemma}\label{lem:conditional_convergence}
The finite-dimensional distributions of $\{u(X_\alpha(u^{-2/\alpha}t)-u) \mid X_\alpha(0)>u, t\in\R\}$ converge weakly to the finite-dimensional distributions of $\{Z_\alpha(t)+\eta, t\in\R\}$, where $\eta$ is a random variable independent of $\{Z_\alpha(t), t\in\R\}$ following a standard exponential distribution.
\end{lemma}
The result in Lemma~\ref{lem:conditional_convergence} is well-known; see,
 e.g., \cite[Lemma~2]{MR2685014}, where the convergence of finite-dimensional distributions is
established on $t\in\R_+$. The extension to $t\in\R$ is straightforward.

\begin{proof}[Proof of Lemma~\ref{lem:pickands_diff}]
The following proof is very similar in its flavor to the proof of \cite[Lemma~3.2]{bisewski2021harmonic}. From \cite[Lemma~9.2.2]{20lectures} and the classical definition of Pickands constant it follows that for any $\alpha\in (0,2)$ and
$\delta\ge 0$
\begin{align*}
\mathcal H_\alpha^{\delta} = \lim_{T\to\infty}\lim_{u\to\infty} \frac{\pk{\sup_{t\in
[0,T]_\delta} X_\alpha(u^{-2/\alpha}t) > u}}{T\Psi(u)},
\end{align*}
where $\Psi(u)$ is the complementary cdf (tail) of the standard normal distribution and $\{X_\alpha, t\in\R\}$ is the process introduced above Eq.~\eqref{eq:X_alpha_cov_assumption}. Therefore,
\begin{align*}
\mathcal H^{\delta/2}_{\alpha} - \mathcal H^{\delta}_{\alpha}
& = \lim_{T\to\infty}\lim_{u\to\infty}
\frac{\pk{\max_{t\in[0,T]_{\delta/2}} X_\alpha(u^{-2/\alpha}t) > u,
\max_{t\in[0,T]_\delta} X_\alpha(u^{-2/\alpha}t) < u}}{T\Psi(u)}.
\end{align*}
Now, notice that we can decompose the event in the numerator above into a sum of disjoint events
\begin{align*}
& \frac{1}{\Psi(u)}\pk{\displaystyle\max_{t\in[0,T]_{\delta/2}}
X_\alpha(u^{-2/\alpha}t) > u, \max_{t\in[0,T]_\delta} X_\alpha(u^{-2/\alpha}t) < u} \\
& = \sum_{\tau\in[0,T]_{\delta/2}}
\pk{\max_{t\in [0,T]_{\delta/2}} X_\alpha(u^{-2/\alpha}t) \leq X_\alpha(u^{-2/\alpha}\tau), \max_{t\in
[0,T]_{\delta}} X_\alpha(u^{-2/\alpha}t) \leq u \mid  X_\alpha(u^{-2/\alpha}\tau) > u}.
\end{align*}
Using the stationarity of the process $X_\alpha$, the above is equal to
\begin{align*}
& \sum_{\tau\in [0,T]_{\delta/2}}
\pk{\max_{t\in[0,T]_{\delta/2}}
X_\alpha(u^{-2/\alpha}(t-\tau)) \leq X_\alpha(0),
\max_{t\in[0,T]_{\delta}} X_\alpha(u^{-2/\alpha}(t-\tau)) \leq u \mid  X_\alpha(0) > u}.
\end{align*}
Applying Lemma~\ref{lem:conditional_convergence} to each element of the sum above, we find that the sum above converges to $\sum_{\tau\in[0,T]_{\delta/2}} U(\tau,T)$, as $u\to\infty$, where
\begin{align*}
U(\tau,T) :=  \pk{\max_{t\in[0,T]_{\delta/2}} Z_\alpha(t-\tau) \leq 0, \max_{t\in\delta\Z\cap[0,T]} Z_\alpha(t-\tau) + \eta \leq 0}.
\end{align*}
We have now established that $\mathcal H_\alpha^{\delta/2}-\mathcal H_\alpha^\delta = \lim_{T\to\infty}\frac{1}{T}\sum_{\tau\in[0,T]_{\delta/2}}U(\tau,T)$. Clearly,
\[U(\tau,\infty) = U(0,\infty) = \pk{\max_{t\in(\delta/2)\Z\setminus\{0\}} Z_\alpha(t) < 0, \max_{t\in\delta\Z\setminus\{0\}} Z_\alpha(t - \tfrac{\delta}{2}\cdot{\rm sgn}(t)) + \eta < 0}.\]
We will now show that $\mathcal H_\alpha^{\delta/2}-\mathcal H_\alpha^{\delta}$ is lower bounded and upper bounded by $\delta^{-1}U(0,\infty)$, which will complete the proof. For the lower bound see that
\begin{equation*}
\mathcal H_\alpha^{\delta/2}-\mathcal H_\alpha^\delta \geq \lim_{T\to\infty}\frac{1}{T}\sum_{\tau\in[0,T]_{\delta/2}} U(\tau,\infty),
\end{equation*}
where the limit is equal to $\delta^{-1}U(0,\infty)$ because the sum above has $[T(\delta/2)^{-1}]$ elements, half of which is equal to $0$ and the other half is equal to $U(0,\infty)$. In order to show the upper bound consider $\ep>0$. For any $\tau\in(\ep T,(1-\ep)T)_{\delta/2}$ we have
\begin{align*}
U(\tau,T) \leq \overline U(T,\ep) :=
\pk{\max_{t\in(-\ep T,\ep T)_{\delta/2}}
Z_\alpha(t) \leq 0, \max_{t\in(-\ep T,\ep T)_\delta} Z_\alpha(t) + \eta \leq 0}.
\end{align*}
Furthermore, we have the following decomposition
\begin{align*}
\mathcal H_\alpha^{\delta/2}-\mathcal H_\alpha^\delta  & =  \lim_{T\to\infty}\frac{1}{T}\left(\sum_{\tau\in(\delta/2)\Z\cap I_-} U(\tau,T) + \sum_{\tau\in(\delta/2)\Z \cap I_0} U(\tau,T) + \sum_{\tau\in(\delta/2)\Z\cap I_+} U(\tau,T)\right),
\end{align*}
where $I_- := [0,\ep T]$, $I_0 := (\ep T,(1-\ep)T)$, $I_+ := [(1-\ep),T]$. The first and the last sum can be bounded by their number of elements $[\ep T(\delta/2)^{-1}]$ because $U(\tau,T)\leq 1$. The middle sum can be bounded by $\frac{1}{2}\cdot [(1-2\ep)T(\delta/2)^{-1}]\overline U(T,\ep)$ because half of its elements are equal to 0 and the other half can be upper bounded by $\overline U(T,\ep)$. After passing with $T\to\infty$ this gives us
\begin{align*}
\mathcal H_\alpha^{\delta/2}-\mathcal H_\alpha^\delta\leq 4\ep\delta^{-1} + (1-2\ep)\delta^{-1}U(0,\infty)
\end{align*}
because $\overline U(T,\ep) \to U(0,\infty)$, as $T\to\infty$. Finally, after passing $\ep\to0$ we obtain the desired result.
\end{proof}

\begin{proof}[Proof of Lemma~\ref{lem:pickands_diff_bounds}]
In the light of Lemma~\ref{lem:pickands_diff}, it suffices to show that
\begin{align*}
\pk{\sup_{t\in\delta\Z\setminus\{0\}} Z_\alpha(t - \tfrac{\delta}{2}\cdot{\rm sgn}(t)) + \eta < 0} \leq \pk{\sup_{t\in\delta\Z\setminus\{0\}} Z_\alpha(t) + \eta < 0}.
\end{align*}
The left-hand side of the above equals to
\begin{align*}
& \pk{\sup_{t\in\delta\Z\setminus\{0\}} \sqrt{2}B_\alpha(t - \tfrac{\delta}{2}\cdot{\rm sgn}(t)) - |t - \tfrac{\delta}{2}\cdot{\rm sgn}(t)|^\alpha + \eta < 0} \\
& \qquad = \pk{\sup_{t\in\delta\Z\setminus\{0\}} \frac{\sqrt{2}B_\alpha(t - \tfrac{\delta}{2}\cdot{\rm sgn}(t))}{|t - \tfrac{\delta}{2}\cdot{\rm sgn}(t)|^{\alpha/2}} - |t - \tfrac{\delta}{2}\cdot{\rm sgn}(t)|^{\alpha/2} + \frac{\eta}{|t - \tfrac{\delta}{2}\cdot{\rm sgn}(t)|^{\alpha/2}} < 0} \\
& \qquad \leq \pk{\sup_{t\in\delta\Z\setminus\{0\}} \frac{\sqrt{2}B_\alpha(t - \tfrac{\delta}{2}\cdot{\rm sgn}(t))}{|t - \tfrac{\delta}{2}\cdot{\rm sgn}(t)|^{\alpha/2}} - |t|^{\alpha/2} + \frac{\eta}{|t|^{\alpha/2}} < 0}.
\end{align*}
Now, for all $t,s\in\delta\Z\setminus\{0\}$ it holds that
\begin{equation}\label{eq:to_show:slepian}
\cov\left(\frac{\sqrt{2}B_\alpha(t - \tfrac{\delta}{2}\cdot{\rm sgn}(t))}{|t - \tfrac{\delta}{2}\cdot{\rm sgn}(t)|^{\alpha/2}},  \frac{\sqrt{2}B_\alpha(s - \tfrac{\delta}{2}\cdot{\rm sgn}(s))}{|s - \tfrac{\delta}{2}\cdot{\rm sgn}(s)|^{\alpha/2}} \right) \leq \cov\left(\frac{\sqrt{2}B_\alpha(t)}{|t|^{\alpha/2}}, \frac{\sqrt{2}B_\alpha(s)}{|s|^{\alpha/2}}\right),
\end{equation}
which is shown below. Since for $t=s$, the covariances in Eq.~\eqref{eq:to_show:slepian} are equal, we may use Slepian lemma \cite[Lemma~2.1.1]{20lectures} and obtain
\begin{align*}
\pk{\sup_{t\in\delta\Z\setminus\{0\}} \frac{\sqrt{2}B_\alpha(t - \tfrac{\delta}{2}\cdot{\rm sgn}(t))}{|t - \tfrac{\delta}{2}\cdot{\rm sgn}(t)|^{\alpha/2}} - |t|^{\alpha/2} + \frac{\eta}{|t|^{\alpha/2}} < 0} & \leq \pk{\sup_{t\in\delta\Z\setminus\{0\}} \frac{\sqrt{2}B_\alpha(t)}{|t|^{\alpha/2}} - |t|^{\alpha/2} + \frac{\eta}{|t|^{\alpha/2}} < 0}
\end{align*}
and the claim follows. It is left to show Eq.~\eqref{eq:to_show:slepian}. Let $t,s\in\R$ be fixed and let
\begin{align*}
c(\delta;t,s) := \cov\left(\frac{B_\alpha(t + \delta\cdot{\rm sgn}(t))}{|t + \delta\cdot{\rm sgn}(t)|^{\alpha/2}},  \frac{B_\alpha(s + \delta\cdot{\rm sgn}(s))}{|s + \delta\cdot{\rm sgn}(s)|^{\alpha/2}} \right).
\end{align*}
We will show that $\delta\mapsto c(\delta;t,s)$ is a non-decreasing function, which will conclude the proof. We have
\begin{align*}
c(\delta;t,s) = \frac{|t + \delta\cdot{\rm sgn}(t)|^\alpha + |s + \delta\cdot{\rm sgn}(s)|^\alpha - |t-s + \delta\cdot({\rm sgn}(t)-{\rm sgn}(s))|^\alpha}{2|s + \delta\cdot{\rm sgn}(s)|^{\alpha/2}\cdot |t + \delta\cdot{\rm sgn}(t)|^{\alpha/2}}.
\end{align*}
We consider two cases: (i) $t,s>0$, and (ii) $s<0<t$. Consider case (i) first. Without loss of generality we assume that $t\geq s$, then
\begin{equation*}
c(\delta;t,s) = \frac{(t + \delta)^\alpha + (s + \delta)^\alpha - (t-s)^\alpha}{2((s + \delta)(t + \delta))^{\alpha/2}}.
\end{equation*}
It suffices to show that the first derivative of $\delta\mapsto c(\delta;t,s)$ is nonnegative. We have
\begin{align*}
\frac{\partial}{\partial\delta} c(\delta,t,s) = \frac{\alpha(1-x)(t+\delta)^{\alpha/2}/4}{(s+\delta)^{1+\alpha/2}} \cdot \bigg( (1-x)^{\alpha-1}(1+x) - x^{\alpha}-1) \bigg),
\end{align*}
where $x := \frac{s+\delta}{t+\delta} \in (0,1]$. The derivative above is nonnegative iff $G_1(x,\alpha) := (1-x)^{\alpha-1}(1+x) + x^{\alpha} - 1 \geq 0$ for all $x\in(0,1]$. It is easy to see that, for any fixed $x\in(0,1]$, $\alpha\mapsto G_1(x,\alpha)$ is a non-decreasing function; this observation combined with the fact that  $G(x,2)=0$ completes the proof of case (i).
In case (ii) we need to show that
\begin{align*}
c(\delta;t,-s) =\frac{(t + \delta)^\alpha + (s + \delta)^\alpha - (t+s + 2\delta)^\alpha}{2((s+\delta)(t+\delta))^{\alpha/2}}
\end{align*}
is a non-decreasing function of $x$ for any $s,t>0$. Without the loss of generality let $0<s\leq t$. Again, we take the first derivative of the above and see that
\begin{align}
\nonumber& \frac{\partial}{\partial\delta} c(\delta,t,s) = \frac{\alpha(1-x)(t+\delta)^{\alpha/2}/4}{(s+\delta)^{1+\alpha/2}} \cdot \bigg( (1-x)(1+x)^{\alpha-1} + x^\alpha - 1) \bigg),
\end{align}
where $x := \frac{s+\delta}{t+\delta}\in(0,1]$. The derivative above is non-negative iff $G_2(x,\alpha) := (1-x)(1+x)^{\alpha-1} + x^\alpha - 1 \geq 0$. Notice that $G_2(x,1) = 0$. We will now show that $\frac{\partial}{\partial\alpha}G_2(x,\alpha) \leq 0$ for all $\alpha\in[0,1]$ and $x\in(0,1]$, which will conclude the proof. We have
\begin{align*}
\frac{\partial}{\partial\alpha}G_2(x,\alpha) & = (1-x)(1+x)^{\alpha-1}\log(x+1)+x^\alpha\log x \\
& \leq (1-x)x^{\alpha-1}\log(x+1)+x^\alpha\log x \\
& = x^{\alpha-1}
((1-x)\log(x+1)+x\log x) \\
& \leq x^{\alpha-1}((1-x)x + x(x-1)) = 0,
\end{align*}
where in the last line we used the fact that $\log(1+x)\leq x$ for all $x>-1$.
\end{proof}


We will now layout preliminaries necessary to prove Proposition~\ref{prop:DY_plus_eta}. First, let us introduce notation, which will be used until the end of this section. For any $\delta>0$, $\lambda>0$ let
\begin{equation}\label{def:p_and_q}
\begin{split}
& p(\delta) := \pk{A(\delta)}, \quad \text{ with } \quad A(\delta) := \{Z_\alpha(-\delta) < 0,Z_\alpha(\delta) < 0\}, \text{ and} \\
& q(\delta, \lambda) := \pk{A(\delta,\lambda)}, \quad \text{ with } \quad A(\delta,\lambda) := \{Z_\alpha(-\delta) + \lambda^{-1}\eta < 0,Z_\alpha(\delta) + \lambda^{-1}\eta < 0\}.
\end{split}
\end{equation}
For any $\delta>0$ and $\lambda>0$ we define the densities of two-dimensional vectors $(Z_\alpha(-\delta), Z_\alpha(\delta))$ and $(Z_\alpha(-\delta) + \lambda^{-1}\eta, Z_\alpha(\delta)+ \lambda^{-1}\eta)$ respectively, with $\textbf{x} := \left(\begin{smallmatrix} x_1\\x_2 \end{smallmatrix}\right) \in\R^2$,
\begin{align}\label{def:unconditional_densities_f}
\begin{split}
f(\textbf{x}; \delta) & := \frac{\pk{Z_\alpha(-\delta)\in{\rm d}x_1,Z_\alpha(\delta)\in{\rm d}x_2}}{{\rm d}x_1{\rm d}x_2}, \\
g(\textbf{x}; \delta, \lambda) & := \frac{\pk{ Z_\alpha(-\delta) + \lambda^{-1}\eta \in{\rm d}x_1, Z_\alpha(\delta) + \lambda^{-1}\eta \in{\rm d}x_2}}{{\rm d}x_1{\rm d}x_2},
\end{split}
\end{align}
as well as the densities of these random vectors conditioned to take negative values on both coordinates, that is
\begin{align}\label{def:conditional_densities_f}
\begin{split}
f^-(\textbf{x}; \delta) & := \frac{\pk{ Z_\alpha(-\delta) \in{\rm d}x_1, Z_\alpha(\delta) \in{\rm d}x_2 \mid A(\delta)}}{{\rm d}x_1{\rm d}x_2},\\
g^-(\textbf{x}; \delta, \lambda) & := \frac{\pk{ Z_\alpha(-\delta) + \lambda^{-1}\eta \in{\rm d}x_1, Z_\alpha(\delta) + \lambda^{-1}\eta \in{\rm d}x_2 \mid A(\delta,\lambda)}}{{\rm d}x_1{\rm d}x_2}.
\end{split}
\end{align}
Now, let $\Sigma$ be the covariance matrix of $(Z_\alpha(-1), Z_\alpha(1))$, that is
\begin{equation}\label{eq:def_Sigma}
\Sigma := \begin{pmatrix}
\cov(Z_\alpha(-1), Z_\alpha(-1)) \  & \ \cov(Z_\alpha(-1), Z_\alpha(1)) \\
\cov(Z_\alpha(1), Z_\alpha(-1)) \ & \ \cov(Z_\alpha(1), Z_\alpha(1))
\end{pmatrix}
= \begin{pmatrix}
2 & 2-2^\alpha \\
2-2^\alpha & 2
\end{pmatrix}.
\end{equation}
By the self-similarity property of fBm, the covariance matrix $\Sigma(\delta)$ of $(Z_\alpha(-\delta), Z_\alpha(\delta))$ equals $\Sigma(\delta) = \delta^{\alpha}\Sigma$. With ${\vk 1}_2 = \left(\begin{smallmatrix} 1\\1 \end{smallmatrix}\right)$ we define:
\begin{equation}\label{def:a_b_c}
a(\textbf{x}) := \textbf{x}^\top \Sigma^{-1}\textbf{x}, \quad b(\textbf{x}) := \textbf{x}^\top\Sigma^{-1}{\vk 1}_2, \quad c := {\vk 1}_2^\top\Sigma^{-1}{\vk 1}_2,
\end{equation}
so that, with $|\Sigma|$ denoting the determinant of matrix $\Sigma$ we have
\begin{align*}
f(\textbf{x}; \delta) & = \frac{1}{2\pi|\Sigma|^{1/2}\delta^{\alpha}} \exp\left\{-\frac{(\textbf{x}+{\vk 1}_2\delta^\alpha)^\top\Sigma(\delta)^{-1}(\textbf{x}+{\vk 1}_2\delta^\alpha)}{2\delta^{\alpha}}\right\}\\
& = \frac{1}{2\pi|\Sigma|^{1/2}\delta^{\alpha}} \exp\left\{-\frac{a(\textbf{x}) + 2b(\textbf{x})\delta^{\alpha} + c\delta^{2\alpha}}{2\delta^{\alpha}}\right\}.
\end{align*}

\begin{lemma}\label{lem:q(delta,lambda)}
For any $\lambda>0$ there exist $\mathcal C_0,\mathcal C_1>0$ such that
\begin{align*}
\mathcal C_0\delta^{\alpha/2} \leq q(\delta,\lambda) \leq \mathcal C_1\delta^{\alpha/2}
\end{align*}
for all $\delta>0$ sufficiently small.
\end{lemma}
\begin{proof}
For the lower bound see that
\begin{align*}
q(\delta, \lambda) & = \pk{\sqrt{2}B_\alpha(-\delta) -\delta^\alpha+ \lambda^{-1}\eta < 0, \sqrt{2}B_\alpha(\delta) -\delta^\alpha+ \lambda^{-1}\eta < 0)}\\
& \geq \pk{\sqrt{2}B_\alpha(-\delta) + \lambda^{-1}\eta < 0, \sqrt{2}B_\alpha(\delta) + \lambda^{-1}\eta < 0, \lambda^{-1}\eta < \delta^{\alpha/2}}\\
& \geq \pk{\sqrt{2}B_\alpha(-\delta) < -\delta^{\alpha/2}, \sqrt{2}B_\alpha(\delta) < -\delta^{\alpha/2}, \lambda^{-1}\eta < \delta^{\alpha/2}}\\
& = \pk{\sqrt{2}B_\alpha(-1) > 1, \sqrt{2}B_\alpha(1) > 1}\pk{\lambda^{-1}\eta < \delta^{\alpha/2}} \\
& = \pk{\sqrt{2}B_\alpha(-1) > 1, \sqrt{2}B_\alpha(1) > 1} \cdot \left(1-\exp\left\{-\lambda\delta^{\alpha/2}\right\}\right),
\end{align*}
which behaves like $\lambda\delta^{\alpha/2}\pk{\sqrt{2}B_\alpha(-1) > 1, \sqrt{2}B_\alpha(1) > 1}$, as $\delta\downarrow0$. For the upper bound see that
\begin{align*}
q(\delta, \lambda) & \leq \pk{Z_\alpha(\delta) + \lambda^{-1}\eta < 0} = \pk{\delta^{\alpha/2}\sqrt{2}B_\alpha(1) - \delta^\alpha + \lambda^{-1}\eta < 0} \\
& = \int_0^\infty \Phi\left(\tfrac{\delta^\alpha-z}{\sqrt{2}\delta^{\alpha/2}}\right)\lambda e^{-\lambda z}{\rm d}z \leq \sqrt{2}\delta^{\alpha/2}\lambda\int_0^\infty\Phi\left(\tfrac{\delta^{\alpha/2}}{\sqrt{2}}-z\right){\rm d}z \\
& \leq \sqrt{2}\delta^{\alpha/2}\lambda\left(\tfrac{\delta^{\alpha/2}}{\sqrt{2}} + \int_{0}^{\infty}\Psi(z){\rm d}z\right) =  \delta^{\alpha/2}\lambda\left(\delta^{\alpha/2} + \frac{\mathbb E|\mathcal N(0,1)|}{\sqrt{2}}\right),
\end{align*}
where $\Phi(\cdot), \Psi(\cdot)$ are the cdf and complementary cdf of the standard normal distribution, respectively. This concludes the proof.
\end{proof}

In the following lemma, we establish the formulas for $f^-$ and $g^-$; most notably, we show that $g^-$ can be upper bounded by $f^-$ uniformly in $\delta$, up to a positive constant.

\begin{lemma}\label{lem:conditional_densities}
For any $\lambda>0$,
\begin{itemize}
\item[(i)] $\displaystyle f^-({\rm\bf x}; \delta) = p(\delta)^{-1} f({\rm\bf x};\delta)$;
\item[(ii)] $\displaystyle g^-({\rm\bf x}; \delta, \lambda) = q(\delta, \lambda)^{-1}f({\rm\bf x};\delta) \int_0^\infty \lambda\exp\left\{-\frac{cz^2  + 2z((\lambda-c)\delta^{\alpha} - b({\rm\bf x}))}{2\delta^{\alpha}}\right\}{\rm d}z$;
\item[(iii)] there exists $C>0$ depending only on $\lambda$, such that for all $\delta$ small enough:\\ $\displaystyle g^-({\rm\bf x}; \delta,\lambda) \leq C f^-({\rm\bf x};\delta)$, for all ${\rm\bf x}\leq 0$;
\end{itemize}
\end{lemma}

\begin{proof}
Part (i) follows directly from the definition. For part (ii),
for $\textbf{x}\leq0$ we have
\begin{align*}
g^-(\textbf{x};\delta,\lambda) & = q(\delta, \lambda)^{-1}\int_0^\infty f(\textbf{x}-{\vk 1}_2z; \delta)\cdot \lambda{\rm e}^{-\lambda z} {\rm d}z \\
& = \int_0^\infty\frac{q(\delta, \lambda)^{-1}}{2\pi|\Sigma|\delta^{\alpha}}\exp\left\{-\frac{(\textbf{x}+{\vk 1}_2(\delta^\alpha-z))^\top\Sigma^{-1}(\textbf{x}+{\vk 1}_2(\delta^\alpha-z))}{2\delta^{\alpha}}\right\}\cdot \lambda{\rm e}^{-\lambda z} {\rm d}z \\
& = \int_0^\infty\frac{\lambda q(\delta, \lambda)^{-1}}{2\pi|\Sigma|\delta^{\alpha}}\exp\left\{-\frac{a(\textbf{x}) + 2b(\textbf{x})(\delta^\alpha-z) + c(\delta^{2\alpha}-2\delta^\alpha z + z^2) + 2\lambda\delta^\alpha z}{2\delta^{\alpha}}\right\} {\rm d}z \\
& = q(\delta, \lambda)^{-1}f(\textbf{x}; \delta)\int_0^\infty \lambda \exp\left\{-\frac{cz^2 + 2z((\lambda-c)\delta^\alpha - b(\textbf{x}))}{2\delta^{\alpha}}\right\} {\rm d}z.
\end{align*}
For part (iii), we have $\Sigma^{-1} = \big(2^\alpha(4-2^\alpha)\big)^{-1} \cdot \left(\begin{smallmatrix}2 & 2^\alpha-2\\ 2^\alpha-2 & 2 \end{smallmatrix}\right)$, thus $\Sigma^{-1}{\vk 1}_2 \geq \textbf{0}$ element-wise. It then follows that $b(\textbf{x}) \leq 0$ when $\textbf{x}\leq0$, which yields the following upper bound
\begin{equation*}
g^-(\textbf{x};\delta,\lambda) \leq q(\delta, \lambda)^{-1}f(\textbf{x}; \delta)\int_0^\infty \lambda \exp\left\{-\frac{cz^2 + 2z(\lambda-c)\delta^\alpha}{2\delta^{\alpha}}\right\} {\rm d}z.
\end{equation*}
Now, it is easy to see that for all $\delta>0$ small enough
\begin{align*}
\int_0^\infty \lambda \exp\left\{-\frac{cz^2 + 2z(\lambda-c)\delta^\alpha}{2\delta^{\alpha}}\right\} {\rm d}z \leq \int_0^\infty \lambda \exp\left\{-cz^2 + 2z(\lambda-c)\right\} {\rm d}z < \infty,
\end{align*}
hence, using part~(i), there exists $\mathcal C>0$ such that
\begin{equation*}
g^-(\textbf{x};\delta,\lambda) \leq \mathcal C \delta^{\alpha/2}q(\delta, \lambda)^{-1}p(\delta)^{-1}f^-(\textbf{x}; \delta).
\end{equation*}
The proof is concluded by noting that $p(\delta) \to \pk{B_\alpha(-1)<0, B_\alpha(1)<0} > 0$ and $\delta^{\alpha/2}q(\delta,\lambda)^{-1} = O(1)$, due to Lemma~\ref{lem:q(delta,lambda)}.
\end{proof}

Recall the definition of $\Sigma$ in Eq.~\eqref{eq:def_Sigma}. In what follows, we put, for $k\in\Z$
\begin{equation}\label{def:c-_c+}
\begin{split}
{\small\begin{pmatrix}
c^-(k) \\
c^+(k)
\end{pmatrix}} & := \Sigma^{-1} \cdot {\small\begin{pmatrix}
\cov(Z_\alpha(k), Z_\alpha(-1))\\
\cov(Z_\alpha(k), Z_\alpha(1))
\end{pmatrix}}\\
& = \frac{1}{2^\alpha(4-2^\alpha)} \cdot {\small\begin{pmatrix}2 & 2^\alpha-2\\ 2^\alpha-2 & 2 \end{pmatrix}} \cdot {\small\begin{pmatrix}
k^\alpha + 1 - (k + {\rm sgn}(k))^\alpha\\
k^\alpha + 1 - (k - {\rm sgn}(k))^\alpha
\end{pmatrix}}.
\end{split}
\end{equation}

\begin{lemma}\label{lem:properties_of_c} For $k\in\Z\setminus\{0\}$,
\begin{itemize}
\item[(i)] $\displaystyle (2-2^{\alpha-1})^{-1} < c^-(k) + c^+(k) \leq 1$ when $\alpha\in(0,1)$;
\item[(ii)] $1 \leq c^-(k) + c^+(k) \leq (2-2^{\alpha-1})^{-1}$ when $\alpha\in(1,2)$;
\end{itemize}
\end{lemma}

\begin{proof}
After some algebraic transformations, from \eqref{def:c-_c+}, we find that
\begin{align*}
c^-(k) + c^+(k) = \frac{2 + 2|k|^\alpha -
(|k|-1)^\alpha -(|k|+1)^\alpha}{4-2^\alpha}, \quad k\in\Z\setminus\{0\}
\end{align*}
Let $f(x)  = 2x^{\alpha}-(x-1)^\alpha-(x+1)^\alpha$.
We have for $x>1$
\bqny{
f'(x) &=& \alpha x^{\alpha-1}\left(2-(1-1/x)^{\alpha-1}-(1+1/x)^{\alpha-1} \right)
\\&=&
\alpha x^{\alpha-1}\left[ 2 - \sum_{n=0}^\infty (-1)^n \frac{(\alpha-1)\cdots(\alpha-n)}{n!}x^{-n} - \sum_{n=0}^\infty \frac{(\alpha-1)\cdots(\alpha-n)}{n!}x^{-n}\right]
\\&=&
\alpha x^{\alpha-3}(\alpha-1)(2-\alpha)\left[\frac{1}{2!} +  \sum_{n=1}^\infty \frac{(\alpha-3)\cdots(\alpha-2(n+1))}{(2(n+1))!}x^{-2n}\right].
}
We see that each of the terms of the sum above is positive, so ${\rm sgn}(f'(x)) = {\rm sgn}(\alpha-1)$. Thus, $f'(x)$ is negative for $\alpha\in (0,1)$ and positive for $\alpha\in(1,2)$. Finally, since
\bqny{
 \lim_{x\to \infty} \frac{2 + 2x^\alpha -
(x-1)^\alpha -(x+1)^\alpha}{4-2^\alpha} = \frac{2}{4-2^{\alpha}}
= (2-2^{\alpha-1})^{-1}
}
and $c^-(1) + c^+(1) =1$, the claim follows.
\end{proof}

We are now ready to prove Proposition~\ref{prop:DY_plus_eta}. In what follows, for any $\delta>0$ and $t\in\R$ let
\begin{align}
\label{def:Y}Y^\delta_\alpha(t) & := Z_\alpha(t) - \E{Z_\alpha(t) \mid (Z_\alpha(-\delta),Z_\alpha(\delta))} \\
\nonumber& = Z_\alpha(t) - \big(c^-(k)Z_\alpha(-\delta) + c^+(k)Z_\alpha(\delta)\big).
\end{align}
It is a well known fact that $\{Y_\alpha^\delta(t), t\in\R\}$ is independent of $(Z_\alpha(-\delta),Z_\alpha(\delta))$.

\begin{proof}[Proof of Proposition~\ref{prop:DY_plus_eta}]
Recall the definition of the events $A(\delta,\lambda)$ and $A(\delta)$ in  \eqref{def:p_and_q}. We have
\begin{align*}
& \pk{\sup_{t\in\delta\Z\setminus\{0\}} Z_\alpha(t) + \eta \leq 0}\\
& \qquad\qquad = \pk{\sup_{k\in\Z\setminus\{-1,0,1\}} Y^\delta_\alpha(\delta k) + \Big(c^-(k)Z_\alpha(-\delta) + c^+(k)Z_\alpha(\delta)\Big) + \eta < 0, A(\delta, 1)},
\end{align*}
with $Y_\alpha(t)$ defined in \eqref{def:Y}. Let $\lambda^* := 1$ when $\alpha\in(0,1]$, and $\lambda^* := (2-2^{\alpha-1})^{-1}$ when $\alpha\in[1,2)$. Due to Lemma~\ref{lem:properties_of_c} we have $\lambda^*\geq1$ and $c^-(k)+c^+(k)\leq \lambda^*$, thus $A(\delta,1) \subseteq A(\delta,\lambda^*)$  and the display above is upper bounded by
\begin{align*}
& \mathbb P \Bigg\{A(\delta,\lambda^*) < 0, \sup_{k\in\Z\setminus\{-1,0,1\}} Y^\delta_\alpha(\delta k) + \Big(c^-(k)\big(Z_\alpha(-\delta)+\tfrac{\eta}{\lambda^*}\big) + c^+(k)\big(Z_\alpha(\delta) + \tfrac{\eta}{\lambda^*}\big)\Big) < 0\Bigg\}\\
& = q(\delta,\lambda^*) \cdot \mathbb P \Bigg\{\sup_{k\in\Z\setminus\{-1,0,1\}} Y^\delta_\alpha(\delta k) + \Big(c^-(k)\big(Z_\alpha(-\delta)+\tfrac{\eta}{\lambda^*}\big) + c^+(k)\big(Z_\alpha(\delta) + \tfrac{\eta}{\lambda^*}\big)\Big) < 0 \,\Big\vert\, A(\delta,\lambda^*)\Bigg\} \\
& = q(\delta,\lambda^*) \cdot \int \mathbb P \Bigg\{\sup_{k\in\Z\setminus\{-1,0,1\}} Y^\delta_\alpha(\delta k) + \Big(c^-(k)x_1 + c^+(k)x_2\Big) < 0\Bigg\}g^-(\textbf{x}; \delta,\lambda^*){\rm d}\textbf{x},
\end{align*}
with $g^-$ defined in \eqref{def:conditional_densities_f}. Now, using Lemma~\ref{lem:conditional_densities}(iii) we know that there exists $\mathcal C>0$ such that for all $\delta>0$ small enough, the expression above is upper bounded by
\begin{align*}
& q(\delta,\lambda^*) \cdot \int \mathbb P \Bigg\{\sup_{k\in\Z\setminus\{-1,0,1\}} Y^\delta_\alpha(\delta k) + \Big(c^-(k)x_1 + c^+(k)x_2\Big) < 0\Bigg\}\mathcal Cf^-(\textbf{x}; \delta){\rm d}\textbf{x} \\
& = \mathcal C q(\delta,\lambda^*)\mathbb P \Bigg\{\sup_{k\in\Z\setminus\{-1,0,1\}} Y^\delta_\alpha(\delta k) + \Big(c^-(k)Z_\alpha(-\delta) + c^+(k)Z_\alpha(\delta)\Big) < 0 \,\Big\vert\, A(\delta) \Bigg\} \\
& = \mathcal C q(\delta,\lambda^*)p(\delta)^{-1} \mathbb P \Bigg\{A(\delta), \sup_{k\in\Z\setminus\{-1,0,1\}} Y^\delta_\alpha(\delta k) + \Big(c^-(k)Z_\alpha(-\delta) + c^+(k)Z_\alpha(\delta)\Big) < 0\Bigg\} \\
& = \mathcal C q(\delta,\lambda^*)p(\delta)^{-1} \pk{\sup_{t\in\delta\Z\setminus\{0\}} Z_\alpha(t) \leq 0}.
\end{align*}
Finally, we have $\pk{\sup_{t\in\delta\Z\setminus\{0\}} Z_\alpha(t) \leq 0} \sim \delta^{-1}\mathcal H_\alpha$, see \cite[Proposition~4]{MR3217455}. This completes the proof due to Lemma~\ref{lem:q(delta,lambda)} and the fact that $p(\delta) \to \pk{B_\alpha(-1)<0, B_\alpha(1)<0}>0$.
\end{proof}

\subsection{Proof of Theorem~\ref{thm:main}, case $\alpha\in (1,2)$}

The following lemma provides a crucial bound for
$\mathcal H_{\alpha} - \mathcal H_{\alpha}^{\delta}$.

\begin{lemma} \label{lemma_sub_additivity} For sufficiently small
 $\delta>0$ it holds, that
\bqny{
\mathcal H_{\alpha} - \mathcal H_{\alpha}^{\delta}
\le 2\E{\sup_{t \in [0,1]}e^{Z_\alpha(t)}-
\sup_{t \in [0,1]_\delta}e^{Z_\alpha(t)}
}.}
\end{lemma}

\begin{proof}[Proof of Lemma~\ref{lemma_sub_additivity}]
As follows from the proof of \citep[Theorem~1]{continuity_pick_const},
the first equation on p.~12 with
$c_\delta := [1/\delta]\delta$, where $[\cdot]$ is the
integer part of a real number, it holds, that
\bqny{
\mathcal H_{\alpha} - \mathcal H_{\alpha}^{\delta} &\le& c_\delta^{-1}
\E{\sup_{t\in [0,c_\delta]}e^{Z_\alpha(t)}
-\sup_{t\in [0,c_\delta]_\delta}e^{Z_\alpha(t)}}
\\&\le&
2\E{\sup_{t\in [0,c_\delta]}e^{Z_\alpha(t)}
-\sup_{t\in [0,c_\delta]_\delta}e^{Z_\alpha(t)}}
\le
2\E{\sup_{t\in [0,1]}e^{Z_\alpha(t)}
-\sup_{t\in [0,1]_\delta}e^{Z_\alpha(t)}},
}
which completes the proof.
\end{proof}

Now we are ready to prove Theorem~\ref{thm:main}(ii).
\begin{proof}[Proof of Theorem~\ref{thm:main}, $\alpha\in (1,2)$]
Note that for any $ y\le x$ it holds, that $e^x-e^y\le (x-y)e^x$.
Implementing this inequality we find that, for $s,t\in[0,1]$,
\bqny{
\left|e^{Z_\alpha(t)}-e^{Z_\alpha(s)}\right| &\le& e^{\max(Z_\alpha(t),Z_\alpha(s))}
|Z_\alpha(t)-Z_\alpha(s)|
\\&\le&
e^{\sqrt 2\max\limits_{w\in [0,1]}B_\alpha(w)}\left|
\sqrt 2(B_\alpha(t)-B_\alpha(s))-(t^\alpha-s^\alpha)\right|.
}
Next, by Lemma~\ref{lemma_sub_additivity} we have
\begin{align*}
\mathcal H_\alpha - \mathcal H_\alpha^\delta & \leq 2\E{\sup\limits_{t,s\in [0,1],|t-s|\le\delta}\left|e^{Z_\alpha(t)}-e^{Z_\alpha(s)}\right|} \\
& \leq 2\sqrt{2}\E{e^{\sqrt 2\max\limits_{w\in [0,1]}B_\alpha(w)}\sup_{t,s\in [0,1],|t-s|\le\delta}|B_\alpha(t)-B_\alpha(s)|}\\
& \quad + 2\E{e^{\sqrt 2\max\limits_{w\in [0,1]}B_\alpha(w)}
\sup\limits_{t,s\in [0,1],|t-s|\le\delta}|t^\alpha-s^\alpha|}.
\end{align*}
Clearly, the second term is upper-bounded by $\mathcal C_1 \delta$ for all $\delta$ small enough. Using H\"{o}lder inequality, the first term can be bounded by
\begin{equation*}
2\sqrt{2}\E{e^{2\sqrt 2\max\limits_{w\in [0,1]}B_\alpha(w)}}^{1/2}
\E{\left(
\sup\limits_{t,s\in [0,1],|t-s|\le\delta}(B_\alpha(t)-B_\alpha(s))
\right)^2}^{1/2}.
\end{equation*}
The first expectation is finite. The random variable inside the second expectation is called the \emph{uniform modulus of continuity}. From \cite[Theorem~4.2, p.~164]{Minicourse} it follows that there exists $\mathcal C>0$ such that
\begin{equation*}
\E{\left(
\sup\limits_{t,s\in [0,1],|t-s|\le\delta}(B_\alpha(t)-B_\alpha(s))
\right)^2}^{1/2} \leq \mathcal C\delta^{\alpha/2}|\log(\delta)|^{1/2}.
\end{equation*}
This concludes the proof.
\end{proof}

\begin{remark}
Note that the proofs of Lemma~\ref{lemma_sub_additivity} and
Theorem~\ref{thm:main} work also in case $\alpha\in(0,1]$.
\end{remark}

\subsection{Proof of Theorem~\ref{thm:main}, case $\alpha\in\{1,2\}$}

\begin{proof}[Proof of Theorem~\ref{thm:main}, $\alpha=1$]
In the following $\phi$, $\Psi$ stand for the pdf and the survival function of a standard Gaussian random variable and $$v(\eta) := \eta\exp(2\sum_{k=1}^\infty \frac{\Psi(\sqrt {\eta k/2})}{k}), \qquad \eta>0.$$
Before giving the proof we formulate and prove the following
auxiliary lemma

\begin{lemma}\label{lemma_v(eta)_derivative}
It holds, that for any $\eta>0$
\bqny{
v'(\eta)
 = \exp\left(2\sum_{k=1}^\infty \frac{\Psi(\sqrt{\frac{\eta k}{2}})}{k}
 \right)
\left(1-\frac{\sqrt \eta}{2\sqrt \pi}
\sum_{k=1}^\infty \frac{e^{-\frac{\eta k}{4}}}{\sqrt k}\right)
.}
\end{lemma}

\begin{proof}[Proof of Lemma~\ref{lemma_v(eta)_derivative}]
To prove the lemma it is sufficient to show that for any $\eta>0$
\bqny{
\frac{\partial}{\partial \eta}
\left(\sum_{k=1}^\infty \frac{\Psi(\sqrt{\frac{\eta k}{2}})}{k}\right) =
\sum_{k=1}^\infty
\frac{\partial}{\partial \eta}
\left(\frac{\Psi(\sqrt{\frac{\eta k}{2}})}{k}\right)
.}

Take $a,b>0$ such that $\eta\in [a,b]$,
$f(\eta) = \sum_{k=1}^\infty (\Psi(\sqrt{\eta k/2})/k)$
and $f_n(\eta) =
\sum_{k=1}^n (\Psi(\sqrt{\eta k/2})/k), n \in \N$.
According to paragraph 3.1 p. 385 in \citep{Matan2003} to claim the
line above it is enough to show that
\\
\\
1) there exists $\eta_0 \in[a,b]$ such that the sequence
$\{f_n(\eta_0)\}_{n\in \mathbb{N}}$ converges to a finite limit,\\
2) $f'_n(\eta), \eta\in [a,b]$ converge uniformly to some function.\\

The first condition holds
since $\Psi(x)<e^{-x^2/2}$ for $x>0$. For the
second condition we need to prove that uniformly
for all $\eta\in [a,b]$ it holds, that
$\sum_{k=n+1}^\infty f'_k(\eta) \to 0$ as $n \to \infty$.
We have
\bqny{
\sum_{k=n+1}^\infty f'_k(\eta) = \sum_{k=n+1}^\infty
\frac{\phi(\sqrt {\eta k/2})}{2\sqrt {2k\eta}} = \sum_{k=n+1}^\infty
\frac{e^{-\eta k/4}}{4\sqrt {\pi k\eta}}
\le \mathcal Ce^{-\mathcal C_1n} \to 0, \ \ n \to \infty
}
and the claim holds.
\end{proof}

Now we are ready to continue proof of Theorem~\ref{thm:main} for
$\alpha =1$.
By Proposition~\ref{prop:formulas_alpha12}(i), we find that
\bqny{
\mathcal A :=\lim_{\eta \to 0}\frac{\mathcal H_0-\mathcal H_\eta}
{\sqrt \eta} = \lim_{\eta \to 0}\frac{1-1/v(\eta)}{\sqrt \eta}.
}
Since $\mathcal H_\eta = v(\eta)^{-1} \to 1$ as $\eta \to 0$ (see, e.g.,
\citep{MR3217455}) we conclude that $\lim_{\eta\to 0} v(\eta)=1$
and hence $\mathcal A = \lim_{\eta \to 0}\frac{v(\eta)-1}{\sqrt \eta}.$
Implementing the L'H\^{o}pital's rule we obtain by Lemma
\ref{lemma_v(eta)_derivative}
\bqny{\mathcal A =
\lim_{\eta \to 0} \frac{v'(\eta)}{1/(2\sqrt \eta)}
= 2\lim_{\eta \to 0}
\sqrt\eta  \exp\left(
2\sum_{k=1}^\infty \frac{\Psi(\sqrt{\frac{\eta k}{2}})}{k}\right)
\left(1-\frac{\sqrt \eta}{2\sqrt \pi}
\sum_{k=1}^\infty \frac{e^{-\frac{\eta k}{4}}}{\sqrt k}\right).}
Note that by the definition of $v(\eta)$ observation that $
\lim_{\eta\to 0}v(\eta) = 1$
implies
$$\sqrt\eta  \exp
\left(2\sum_{k=1}^\infty \frac{\Psi(\sqrt{\frac{\eta k}{2}})}{k}\right)
\sim \frac{1}{\sqrt \eta}, \ \ \ \eta \to 0$$
and hence
\bqny{
\mathcal A = \lim_{\eta\to 0}\frac{2}{\sqrt \eta}
\Big(1-\frac{\sqrt \eta}{2\sqrt \pi}
\sum_{k=1}^\infty \frac{e^{-\frac{\eta k}{4}}}{\sqrt k}
\Big).
}
Let $x = \sqrt \eta/2$, thus
\bqny{
\mathcal A = \lim_{x\to 0}\frac{1}{x}
\Big(1-\frac{x}{\sqrt \pi}
\sum_{k=1}^\infty \frac{e^{-x^2k}}{\sqrt k}\Big)
\Big) = \lim_{x\to 0}\frac{1}{x}
\Big(1-\frac{x}{\sqrt \pi}\li(e^{-x^2})\Big),
}
where $\li$ is the polylogarithm function, see, e.g.,
\citep{MR2310128}.
As follows from \citep[Eq.~(9.3)]{Polylogarithm_Properties}
\bqny{
\lim_{x\to 0}\frac{1}{x}
\Big(1-\frac{x}{\sqrt \pi}\li(e^{-x^2})\Big) &=&
\lim_{x\to 0}\frac{1}{x}
\Big(1-\frac{x}{\sqrt \pi}\Big(\Gamma(1/2)(x^2)^{-1/2}+
\zeta(1/2)+
\sum_{k=1}^\infty\zeta(1/2-k)\frac{(-x^2)^k}{k!} \Big)\Big)
\\&=&
\frac{\zeta(1/2)}{\sqrt \pi}-\frac{1}{\sqrt \pi}\lim_{x\to 0}\Big(
\sum_{k=1}^\infty\zeta(1/2-k)\frac{x^{2k}(-1)^k}{k!}\Big)
.}
Thus, to prove the claim is it enough to show that
\bqn{\label{sum_zeta_functions}
\lim_{x\to 0}
\sum_{k=1}^\infty\zeta(1/2-k)\frac{x^{2k}(-1)^k}{k!} = 0.
}

By the Riemann functional equation (equation (2.3) in
\citep{MR3978987}) and
observation that $\zeta(s)$ is strictly decreasing for real $s>1$
we have for any
natural number $k$
\bqny{|\zeta(1/2-k)| \le
2^{1/2-k}\pi^{-1/2-k}\Gamma(1/2+k)\zeta(1/2+k)
\le
2^{-k}\Gamma(k+1)\zeta(3/2)
= \frac{\zeta(3/2)k!}{2^k}
.}
Thus, for $|x|<1$ we have
\bqny{
\Big|\sum_{k=1}^\infty\zeta(1/2-k)\frac{x^{2k}(-1)^k}{k!}\Big|
\le x^2 \sum_{k=1}^\infty \frac{|\zeta(1/2-k)|}{k!}
\le
x^2\zeta(3/2)\sum_{k=1}^\infty2^{-k} = \zeta(3/2)x^2
}
and  \eqref{sum_zeta_functions} follows,
this completes the proof of case $\alpha=1$.
\end{proof}

\begin{proof}[Proof of Theorem~\ref{thm:main}, $\alpha=2$]
By Proposition~\ref{prop:formulas_alpha12}(ii) we have, as $\delta\to 0$,
\bqny{
\mathcal H_2-\mathcal H_2^\delta &=&
\frac{1}{\sqrt\pi}-\frac{2}{\delta}\left(
\Phi(\frac{\delta}{\sqrt 2})-\frac{1}{2}\right) = \frac{1}{\sqrt \pi}\left(1- \frac{\sqrt 2}{\delta}
\int^{\delta/\sqrt 2}_0 e^{-x^2/2}{\rm d}x\right) \\
&=& \frac{1}{\sqrt \pi}\left(1- \frac{\sqrt 2}{\delta}
\int^{\delta/\sqrt 2}_0\left(1-\frac{x^2}{2}\right){\rm d}x+O(\delta^4)\right)
= \frac{1}{12\sqrt\pi}
}
and the claim follows.
\end{proof}

\subsection{Proofs of other Results}

\begin{proof}[Proof of Proposition~\ref{prop:formulas_alpha12}]
For (i) consult, e.g., \citep{MR3379923,MR3217426} and for (ii)
see \citep[Eq.~(2.9)]{MR3745388}.
\end{proof}

Let $\delta\geq0$. Define measure $\mu_\delta$ such that for real numbers $a\le b$
\bqny{
\mu_\delta([a,b]) =
\begin{cases}
\delta\cdot \#\{[a,b]_\delta \}, & \delta>0\\
b-a, & \delta = 0.
\end{cases}
}

\begin{proof}[Proof of Theorem~\ref{thm:tail_behavior_pickands}]
We have for any $T>0, x\ge 1$ and $\delta\ge 0$
\bqn{\label{p1(x)_and_p2(x)}
\pk{\xi_\alpha^\delta>x} &\le &
\notag
\pk{\frac {e^{\sup_{t\in \R}Z_\alpha(t)}}{\int_\R e^{Z_\alpha(t)}{\rm d}\mu_\delta}>x}
\\ &=&\notag
\pk{\frac{e^{\sup_{t\in [-T,T]}Z_\alpha(t)}}
{\int_\R e^{Z_\alpha(t)}{\rm d}\mu_\delta}
>x \text{ and } Z_\alpha(t) \text{ achieves its maximum at }  t \in [-T,T] }
\\ \notag & \ &  \ + \
\pk{
\frac{e^{\sup_{t\in \R\backslash [-T,T]}Z_\alpha(t)}}
{\int_\R e^{Z_\alpha(t)}{\rm d}\mu_\delta}
>x \text{ and } Z_\alpha(t)
\text{ achieves its maximum at }  t \in \R\backslash[-T,T] }
\\ \notag & \le &
\pk{
\frac{e^{\sup_{t\in [-T,T]}Z_\alpha(t)}}
{\int_\R e^{Z_\alpha(t)}{\rm d}\mu_\delta}
>x }
+
\pk{\exists t\in \R\backslash [-T,T] : Z_\alpha(t)>0}
\\ \notag &\le&
\pk{
\frac{e^{\sup_{t\in [-T,T]}Z_\alpha(t)}}
{\int_{[-T,T)} e^{Z_\alpha(t)}{\rm d}\mu_\delta}
>x }+
2\pk{\exists t\ge T : Z_\alpha(t)>0}
\\ & =:& p_1(T,x)+2p_2(T).
}

\emph{Estimation of $p_2(T)$. }
By the self-similarity of fBm we have
\bqny{\notag
p_2(T) &\le& \sum_{k=1}^\infty \pk{\exists t \in [kT,(k+1)T]: \sqrt 2
B_\alpha(t)-t^{\alpha}>0}
\\&=& \notag
\sum_{k=1}^\infty \pk{\exists t \in [1,1+\frac{1}{k}]: \sqrt 2
B_\alpha(t)(kT)^{\alpha/2}>(kT)^{\alpha}t^{\alpha}}
\\&\le& \notag
\sum_{k=1}^\infty \pk{\exists t \in [1,2]: \sqrt 2
B_\alpha(t)>(kT)^{\alpha/2}}
}
and thus using Borell-TIS inequality, we find that for all $T\ge 1$
\bqn{\label{estimate_p2(T)}
p_2(T) \le \sum_{k=1}^\infty \mathcal C e^{-\frac{(kT)^{\alpha}}{10}}\le \mathcal C e^{-\frac{T^{\alpha}}{10}}
.}

\emph{Estimation of $p_1(T,x)$.} Observe that for $T,x\ge 1$ and $\delta\in[0,1]$,
\bqny{
p_1(T,x) &\le& \pk{
\frac{ \sum_{k=-T}^{T-1}e^{\sup_{t\in [k,k+1]}Z_\alpha(t)}}
{\sum_{k=-T}^{T-1}\int_{[k,k+1)} e^{Z_\alpha(t)}{\rm d}\mu_\delta}
>x } =: \pk{\frac{\sum_{k=-T}^{T-1}a_k(\omega)}{
\sum_{k=-T}^{T-1}b_k(\omega)}>x}.
}
Since the event $\{\sum_{k=-T}^{T-1}a_k(\omega)/
\sum_{k=-T}^{T-1}b_k(\omega)>x\}$ implies $\{a_k(\omega)/b_k(\omega)>x, \text{ for some } k \in [-T,T-1]_1\}$, thus
\bqny{
\pk{\frac{\sum_{k=-T}^{T-1}a_k(\omega)}{
\sum_{k=-T}^{T-1}b_k(\omega)}>x} \le
\sum_{k=-T}^{T-1}
\pk{\frac{a_k(\omega)}{b_k(\omega)}>x}
\le 2T \sup_{k \in [-T,T]}
\pk{
\frac{\sup_{t\in [k,k+1]}e^{Z_\alpha(t)}}
{\int_{[k,k+1)} e^{Z_\alpha(t)}{\rm d}\mu_\delta}>x
}.}
Therefore, we obtain that for $x,T\ge 1$
\bqny{
p_1(T,x) \le 2T \sup_{k \in [-T,T]}
\pk{\frac{\sup_{t\in [k,k+1]}e^{Z_\alpha(t)}}
{\int_{[k,k+1)} e^{Z_\alpha(t)}{\rm d}\mu_\delta}>x}
.}
Next we have by the stationarity of the increments of fBm for
$x,T\ge 1$
\begin{align*}
& \pk{\frac{\sup_{t\in [k,k+1]}e^{Z_\alpha(t)}}
{\int_{[k,k+1)} e^{Z_\alpha(t)}{\rm d}\mu_\delta}>x
} \leq \pk{\frac{\sup_{t\in [k,k+1]}e^{Z_\alpha(t)}}
{\mu_\delta[k,k+1]\inf_{[k,k+1]} e^{Z_\alpha(t)}}>x} \\
& \qquad\qquad\leq \pk{\exists t,s \in [k,k+1]: Z_\alpha(t)-Z_\alpha(s)>\log(\tfrac{x}{2})} \\
& \qquad\qquad\leq \pk{\exists t,s \in [k,k+1]: B_\alpha(t)-B_\alpha(s)>
\frac{\log(\tfrac{x}{2}) - \sup_{t,s \in [k,k+1]}(|t|^{\alpha}-|s|^{\alpha}) }{\sqrt 2}} \\
& \qquad\qquad\leq \pk{\exists t \in [0,1]: B_\alpha(t)>
\frac{\log x -\mathcal C\max(1,T^{\alpha-1})}{\sqrt 2}},
\end{align*}
where in the second line we used that $\mu_\delta[k,k+1]\geq1/2$ for $\delta\in[0,1]$. Thus, for $T,x\ge 1$
\bqn{\label{bound_p1(m,x)}
p_1(T,x) &\le& 2T\pk{\exists t \in [0,1]: B_\alpha(t)>
\frac{\log x -\mathcal C\max(1,T^{\alpha-1})}{\sqrt 2}}
}
and combining the statement above with \eqref{p1(x)_and_p2(x)} and
\eqref{estimate_p2(T)}
we have for $x,T\ge1$
\bqny{
\pk{\frac{\sup_{t\in \R}e^{Z_\alpha(t)}}
{\int_\R e^{Z_\alpha(t)}{\rm d}\mu_\delta}
>x} \le \widetilde{\mathcal C} e^{-\frac{T^{\alpha}}{10}} +
2T\pk{\exists t \in [0,1]: B_\alpha(t)>
\frac{\log x -\mathcal C\max(1,T^{\alpha-1})}{\sqrt 2}}.
}

Assume that $\alpha\le 1$. Then choosing $T=x$ in the line above
we have by Borell-TIS inequality for all $x\ge 1$
\bqny{
\pk{\frac{\sup_{t\in \R}e^{Z_\alpha(t)}}
{\int_\R e^{Z_\alpha(t)}{\rm d}\mu_\delta}>x} \le  \mathcal C_1e^{-\mathcal C_2\log^2 x}.
}
Assume that $\alpha>1$.
Taking $T = \mathcal C'(\log x)^{\frac{1}{\alpha-1}}$ with sufficiently small $\mathcal C'>0$
we obtain for $x\ge 1$
\bqny{
\pk{\frac{\sup_{t\in \R}e^{Z_\alpha(t)}}
{\int_\R e^{Z_\alpha(t)}{\rm d}\mu_\delta}
>x} \le \widetilde{\mathcal C} e^{-\mathcal C''(\log x)^{\frac{\alpha}{\alpha-1}}
}+\mathcal C_1e^{-\mathcal C_2\log^2 x}\le \mathcal C_3e^{-\mathcal C\log^2 x}}
and the claim follows.
\end{proof}

\begin{proof}[Proof of Theorem~\ref{theorem_truncation}]
Observe that $|x^p-y^p| \le p|x-y|(x^{p-1}+y^{p-1})$ for all $x,y\ge 0$ and $p\ge 1$,
this can be shown straightforwardly by the differentiation.
Hence we have
\bqny{
\lefteqn{\left|\E{(\xi_\alpha^\delta(T))^p} - \E{(\xi_\alpha^\delta)^p}\right|
\le
\E{\left|(\xi_\alpha^\delta(T))^p - (\xi_\alpha^\delta)^p\right|}}\\
\\ &&\le
p\E{\left|\xi_\alpha^\delta(T)-\xi_\alpha^\delta
\right|\cdot\left((\xi_\alpha^\delta(T))^{p-1} + (\xi_\alpha^\delta)^{p-1}
\right)}
\\ &&\le
p\left(\E{\left(\xi_\alpha^\delta(T)-\xi_\alpha^\delta\right)^2}\right)^{1/2}
\cdot
\left(\E{\left((\xi_\alpha^\delta(T))^{p-1} + (\xi_\alpha^\delta)^{p-1}
\right)^2}\right)^{1/2}
\\ &&\le
\sqrt 2 p\left(\E{\left(\xi_\alpha^\delta(T)-\xi_\alpha^\delta\right)^2}\right)^{1/2}
\cdot
\left(\E{
(\xi_\alpha^\delta(T))^{2p-2} + (\xi_\alpha^\delta)^{2p-2}}\right)^{1/2}
\\ &&=:
\sqrt 2 p\left(\E{\beta^2}\right)^{1/2}\cdot \left(\E{\kappa_{p}}
\right)^{1/2}.
}

We have
\bqny{
\beta := \frac{\sup_{t\in\delta \Z}e^{Z_\alpha(t)} }
{\int_{\R}e^{Z_\alpha(t)}{\rm d}\mu_\delta}-
\frac{\sup_{t\in[-T,T]_\delta}e^{Z_\alpha(t)} }
{\int_{[-T,T]}e^{Z_\alpha(t)}{\rm d}\mu_\delta} = \beta_1 - \beta_2\beta_3,
}
where
\bqny{
\beta_1 = \frac{\sup\limits_{t\in\delta\Z}e^{Z_\alpha(t)}-\sup\limits_{t\in[-T,T]_\delta
}e^{Z_\alpha(t)}}{\int_{\R}e^{Z_\alpha(t)}{\rm d}\mu_\delta}\ge 0, \quad \beta_2 = \frac{
\int_{\R\backslash[-T,T]}e^{Z_\alpha(t)}{\rm d}\mu_\delta
}{\int_{\R}e^{Z_\alpha(t)}{\rm d}\mu_\delta}>0, \quad \beta_3 = \frac{\sup\limits_{t\in[-T,T]_\delta}e^{Z_\alpha(t)}}{\int_{[-T,T]}e^{Z_\alpha(t)}{\rm d}\mu_\delta}>0.
}
After applying H\"{o}lder inequality we obtain
\bqn{\label{proof_lemma_truncation_holder}
\E{\beta^2}\le 2\E{\beta_1^2}+2\sqrt{\E{\beta_2^4}\E{\beta_3^4}}.
}
We have by \eqref{bound_p1(m,x)} that for $x,T\ge 1$
\bqny{
\pk{\beta_3>x} \le 2T\pk{\exists t \in [0,1]: B_\alpha(t)>
\frac{\log x -\mathcal C\max(1,T^{\alpha-1})}{\sqrt 2}}}
implying that for $T\ge 1$
\bqn{
\notag\E{\beta_3^4} &=& \int_{0}^\IF\pk{\beta_2>x^{1/4}}{\rm d}x
\\&\le& \label{int_estim}
2T \int_{0}^\IF\pk{\exists t \in [0,1]: B_\alpha(t)>
\frac{\frac{1}{4}\log x -\mathcal C\max(1,T^{\alpha-1})}{\sqrt 2}}{\rm d}x
\\&\le&\notag 2T\left( \int_0^{\exp(5\mathcal C\max(1,T^{\alpha-1}))}1{\rm d}x+
\int_{\exp(5\mathcal C\max(1,T^{\alpha-1}))}^\IF
\pk{\exists t \in [0,1]: B_\alpha(t)>
\mathcal C_3\log x }{\rm d}x\right)
\\&\le& \notag
\mathcal C_1e^{\mathcal C\max(1,T^{\alpha-1})}.
}
Finally for $\alpha\in (0,2)$ and $T\ge 1$ we have
\bqn{\label{proof_lemma_truncation_beta3_up_bound}
\E{\beta_3^4} \le \mathcal C_1e^{\mathcal C\max(1,T^{\alpha-1})}.
}
Next, we focus on properties of $\beta_2$. We have for $k>0$
and sufficiently large $T$
\bqny{
\pk{\int_{[kT,(k+1)T)} e^{Z_\alpha(t)}{\rm d}\mu_\delta >
e^{-\frac{1}{2}T^\alpha k^\alpha}}
&\le&
\pk{(T+1) \sup_{t\in[kT,(k+1)T]}e^{Z_\alpha(t)} >
e^{-\frac{1}{2}T^\alpha k^\alpha}}
\\&=&
\pk{\log (T+1)+ \sup_{t\in[kT,(k+1)T]} Z_\alpha(t) >
 -\frac{1}{2}T^\alpha k^\alpha}
\\&=&
\pk{ \exists t\in[kT,(k+1)T] : \frac{B_\alpha(t)}{t^{\alpha/2}}
 > \frac{t^{\alpha}-\frac{1}{2}T^\alpha k^\alpha-\log (T+1)}{\sqrt 2 t^{\alpha/2}}}
\\&\le&
\pk{ \exists t\in[kT,(k+1)T] : \frac{B_\alpha(t)}{t^{\alpha/2}}
 > (Tk)^{\alpha/2}/3},
}
which, according to Borell-TIS inequality, is upper-bounded by $e^{-k^{\alpha}T^{\alpha}/19}$. By the lines above we obtain that with probability at least
$1-\sum_{k\in \Z\backslash\{0\}}e^{-|k|^{\alpha}T^{\alpha}/19}
\ge 1-e^{-T^{\alpha}/20}$ for large $T$
\bqny{
\int_{\R\backslash[-T,T]}e^{Z_\alpha(t)}{\rm d}\mu_\delta \le
\sum_{k\in \Z\backslash\{0\}}
e^{-\frac{1}{2}T^\alpha |k|^\alpha} \le
e^{-T^\alpha/3}.
}
Combining everything together we obtain that for sufficiently large $T$
\bqn{\label{111}
\pk{\int_{\R\backslash[-T,T]}e^{Z_\alpha(t)}{\rm d}\mu_\delta > e^{-T^\alpha/3} }
\le e^{-T^\alpha/20}.
}
Next we notice that for $T\geq1$
\bqny{
\pk{\int_{[-T,T]} e^{Z_\alpha(t)}{\rm d}\mu_\delta<e^{-\frac{T^\alpha}{4}}} \le
\pk{\int_{[0,1]} e^{Z_\alpha(t)}{\rm d}t\mu_\delta<e^{-\frac{T^\alpha}{4}}} \le
\pk{\sup_{t\in[0,1]} Z_\alpha(t)<- \frac{T^{\alpha}}{4}}
}
so by Borell-TIS inequality, the above is bounded by $e^{-\frac{T^{2\alpha}}{65}}$ for all sufficiently large $T$. This result, in combination with \eqref{111} gives us
$\pk{\beta_2>e^{-T^\alpha/12 }}\le e^{-T^\alpha/21}$
for all sufficiently large $T$.
Thus, since $\beta_2\in [0,1]$ from the line above we immediately obtain that
\bqn{\label{proof_lemma_truncation_beta2_up_bound}
\E{\beta^4_2}\le \mathcal C_1e^{-\mathcal CT^\alpha}.
}
for $T\ge 1$. By \eqref{estimate_p2(T)} we observe, that
\bqny{
\pk{\beta_1>0} \le \pk{\exists t \notin [-T,T]: Z_\alpha(t)>0}
\le 2p_2(T)\le e^{-\mathcal CT^\alpha}.
}
Next by Theorem~\ref{thm:tail_behavior_pickands} we have
for $x\ge 1$ that
\bqny{
\pk{\beta_1>x} \le \pk{
\frac{\sup_{t\in\delta\Z}e^{Z_\alpha(t)}}{\int_{\R}e^{Z_\alpha(t)}{\rm d}\mu_\delta}>x}
\le \mathcal C_1e^{-\mathcal C_2\log^2x}
}
and thus $\E{\beta_1^4}<\mathcal C$ with positive constant $\mathcal C$ that does not depend on $T$. With $A_T:= \{\beta_1(\omega)>0\}$, by H\"{o}lder inequality for large $T$ we have
\bqny{
 \E{\beta_1^2}
= \E{\beta_1^2 \cdot\ind(\Omega_T) }
\le \sqrt{\E{\beta_1^4}}\sqrt{\E{\ind(\Omega_T)}}
\le \mathcal C_1e^{-\mathcal C_2T^{\alpha}}.
}
By the line above and \eqref{proof_lemma_truncation_beta3_up_bound},
\eqref{proof_lemma_truncation_beta2_up_bound} and
\eqref{proof_lemma_truncation_holder}
we obtain for $T\ge 1$
\bqn{\label{proof_lemma_truncation_beta_up_bound}
\E{\beta^2}\le \mathcal C_2e^{-\mathcal C_1T^\alpha}.
}
Our next aim is estimation of $\E{\kappa_{p}}$.
We have for $T\ge 1$ that
\bqny{
\E{(\xi_\alpha^\delta(T))^{2p-2}} &=& \int_0^\IF
\pk{\xi_\alpha^\delta(T)>x^{\frac{1}{2p-2}}}{\rm d}x =\int_0^\IF
\pk{p_1(T,x^{\frac{1}{2p-2}})}{\rm d}x\\
&\le &
\int_{0}^\IF\pk{\exists t \in [0,1]: B_\alpha(t)>
\frac{\log x^{\frac{1}{2p-2}} -\mathcal C\max(1,T^{\alpha-1})}{\sqrt 2}}{\rm d}x.
}
By the same arguments, as in Eq.~\eqref{int_estim}, the last integral above does not
exceed $\mathcal C_1e^{\max \mathcal C_2(T^{\alpha-1},1)}$ and since by Theorem~\ref{thm:tail_behavior_pickands}
$\xi_\alpha^\delta$ has all finite moments uniformly bounded for
all $\delta\ge 0$ we obtain that for $T\ge 1$
$$\kappa_{p} \le \mathcal C_1e^{\max \mathcal C_2(T^{\alpha-1},1)}.$$
Combining the bound above with \eqref{proof_lemma_truncation_beta_up_bound}
we have $\sqrt 2 p\E{\beta^2}\E{\kappa_{p}} \le e^{-\mathcal C_pT^\alpha}$
for sufficiently large $T$ and the claim follows.
\end{proof}


\begin{proof}[Proof of Corollary \ref{corollary_decreasing_Pickands_constant}]
\emph{Case $\alpha=1$.}
First we show that
 $v(\eta)$ is an increasing function for
 $\eta> 0$, that is equivalent with fact that $v'(\eta)>0$ for $\eta>0$.
In the light of Lemma~\ref{lemma_v(eta)_derivative}
it is sufficient to show
\bqny{
\frac{\sqrt \eta}{2\sqrt \pi}
\sum_{k=1}^\infty \frac{e^{-\frac{\eta k}{4}}}{\sqrt k}<1,\quad  \eta>0.
}
We have
\bqny{
\frac{\sqrt\eta}{2\sqrt \pi}
\sum_{k=1}^\infty \frac{e^{-\frac{\eta k}{4}}}{\sqrt k}
 < \frac{1}{\sqrt \pi}
 \sqrt\frac{\eta}{4}\int_0^\infty  e^{-\frac{\eta z}{4}}z^{-1/2}dz =
\frac{1}{\sqrt \pi}\int_0^\infty  e^{-\frac{\eta z}{4}}
(\frac{\eta z}{4})^{-1/2}d(\frac{\eta z}{4}) =\frac{1}{\sqrt \pi}
\Gamma(1/2) = 1
}
and hence $\mathcal{H}^{\eta}_1 =1/v(\eta)$ is decreasing for $\eta>0$.
Since by the classical
definition $\mathcal{H}_{1}^0>\mathcal{H}_1^{\eta}$ for any $\eta>0$ we obtain
the claim.\newline

\emph{Case $\alpha =2$.} We have by Proposition
\ref{prop:formulas_alpha12}(ii)
\bqny{
\mathcal{H}_2^\delta =
 \frac{2}{\delta}\left( \Phi(\delta/\sqrt{2})-\frac{1}{2}   \right)
 = \frac{2}{\delta\sqrt{2\pi}}\int_{0}^{\delta/\sqrt{2}}
 e^{-x^2/2}{\rm d}x =
 \frac{1}{\eta\sqrt{\pi}}\int_{0}^{\eta}
 e^{-x^2/2}{\rm d}x,
}
where $\eta = \delta/\sqrt 2$. The derivative of the last integral
above with respect to $\eta$ equals
\bqny{
 \frac{1}{\sqrt{\pi}}\left(
-\frac{1}{\eta^2}\int_{0}^{\eta} e^{-x^2/2}{\rm d}x
+\frac{1}{\eta}e^{-\eta^2/2}
\right) =
\frac{1}{\sqrt{\pi}\eta^2}\left(
\int_{0}^{\eta} (e^{-\eta^2/2}-e^{-x^2/2}){\rm d}x
\right)<0
}
and the claim follows.
\end{proof}

\section*{Acknowledgements}
We would like to thank Prof. Enkelejd Hashorva and Prof.
Krzysztof D\c{e}bicki for fruitful discussions. We are also greatful to the anonymous reviewer for careful reading and suggesting an improvement of our upper bound in Theorem~\ref{thm:main}(ii).\\
Krzysztof Bisewski's research was funded by SNSF Grant 200021-196888.\\
Grigori Jasnovidov was supported by Ministry of Science and Higher Education of the Russian Federation grant 075-15-2022-289.
\newline

{\bf Data Availability Statement:} The article shares no data.

{\bf Conflict of interest statement:}
The authors declare that they do not have any conflicts of interests.


\begin{thebibliography}{}

\bibitem[Albin and Choi, 2010]{MR2685014}
Albin, J. M.~P. and Choi, H. (2010).
\newblock A new proof of an old result by {P}ickands.
\newblock {\em Electron. Commun. Probab.}, 15:339--345.

\bibitem[Aurzada et~al., 2020]{MR4158797}
Aurzada, F., Buck, M., and Kilian, M. (2020).
\newblock Penalizing fractional {B}rownian motion for being negative.
\newblock {\em Stochastic Process. Appl.}, 130(11):6625--6637.

\bibitem[Bai et~al., 2018]{longKrzys}
Bai, L., D\c{e}bicki, K., Hashorva, E., and Luo, L. (2018).
\newblock On generalised {P}iterbarg constants.
\newblock {\em Methodol. Comput. Appl. Probab.}, 20(1):137--164.

\bibitem[Bisewski et~al., 2021]{bisewski2021harmonic}
Bisewski, K., Hashorva, E., and Shevchenko, G. (2021).
\newblock The harmonic mean formula for random processes.
\newblock {\em arXiv preprint arXiv:2106.11707}.

\bibitem[Bisewski and Ivanovs, 2020]{bisewski2020zooming}
Bisewski, K. and Ivanovs, J. (2020).
\newblock Zooming-in on a {L}\'{e}vy process: failure to observe threshold
  exceedance over a dense grid.
\newblock {\em Electron. J. Probab.}, 25:Paper No. 113, 33.

\bibitem[Borovkov et~al., 2017]{MR3574693}
Borovkov, K., Mishura, Y., Novikov, A., and Zhitlukhin, M. (2017).
\newblock Bounds for expected maxima of {G}aussian processes and their discrete
  approximations.
\newblock {\em Stochastics}, 89(1):21--37.

\bibitem[Borovkov et~al., 2018]{MR3776210}
Borovkov, K., Mishura, Y., Novikov, A., and Zhitlukhin, M. (2018).
\newblock New and refined bounds for expected maxima of fractional {B}rownian
  motion.
\newblock {\em Statist. Probab. Lett.}, 137:142--147.

\bibitem[Cvijovi\'{c}, 2007]{MR2310128}
Cvijovi\'{c}, D. (2007).
\newblock New integral representations of the polylogarithm function.
\newblock {\em Proc. R. Soc. Lond. Ser. A Math. Phys. Eng. Sci.},
  463(2080):897--905.

\bibitem[Dalang et~al., 2009]{Minicourse}
Dalang, R., Khoshnevisan, D., Mueller, C., Nualart, D., and Xiao, Y. (2009).
\newblock {\em A minicourse on stochastic partial differential equations},
  volume 1962 of {\em Lecture Notes in Mathematics}.
\newblock Springer-Verlag, Berlin.
\newblock Held at the University of Utah, Salt Lake City, UT, May 8--19, 2006,
  Edited by Khoshnevisan and Firas Rassoul-Agha.

\bibitem[D\c{e}bicki, 2002]{DE2002}
D\c{e}bicki, K. (2002).
\newblock Ruin probability for {G}aussian integrated processes.
\newblock {\em Stochastic Process. Appl.}, 98(1):151--174.

\bibitem[D\c{e}bicki, 2005]{MR2222683}
D\c{e}bicki, K. (2005).
\newblock Some properties of generalized {P}ickands constants.
\newblock {\em Teor. Veroyatn. Primen.}, 50(2):396--404.

\bibitem[D\c{e}bicki et~al., 2017]{SBK}
D\c{e}bicki, K., Engelke, S., and Hashorva, E. (2017).
\newblock Generalized {P}ickands constants and stationary max-stable processes.
\newblock {\em Extremes}, 20(3):493--517.

\bibitem[D\c{e}bicki and Hashorva, 2017]{MR3745388}
D\c{e}bicki, K. and Hashorva, E. (2017).
\newblock On extremal index of max-stable stationary processes.
\newblock {\em Probab. Math. Statist.}, 37(2):299--317.

\bibitem[D\c{e}bicki et~al., 2015]{ParisRuinGenrealHashorvaJiDebicki}
D\c{e}bicki, K., Hashorva, E., and Ji, L. (2015).
\newblock Parisian ruin of self-similar {G}aussian risk processes.
\newblock {\em J. Appl. Probab.}, 52(3):688--702.

\bibitem[D\c{e}bicki et~al., 2016]{ParisianFiniteHoryzon}
D\c{e}bicki, K., Hashorva, E., and Ji, L. (2016).
\newblock Parisian ruin over a finite-time horizon.
\newblock {\em Sci. China Math.}, 59(3):557--572.

\bibitem[D\c{e}bicki et~al., 2021]{continuity_pick_const}
D\c{e}bicki, K., Hashorva, E., and Michna, Z. (2021).
\newblock On the continuity of {P}ickands constants.
\newblock {\em arXiv preprint arXiv:2105.10435}.

\bibitem[D\c{e}bicki et~al., 2020]{SojournInfty}
D\c{e}bicki, K., Liu, P., and Michna, Z. (2020).
\newblock Sojourn times of {G}aussian processes with trend.
\newblock {\em J. Theoret. Probab.}, 33(4):2119--2166.

\bibitem[D\c{e}bicki and Mandjes, 2015]{MR3379923}
D\c{e}bicki, K. and Mandjes, M. (2015).
\newblock {\em Queues and {L}\'{e}vy fluctuation theory}.
\newblock Universitext. Springer, Cham.

\bibitem[D\c{e}bicki et~al., 2019]{DebZbiXia}
D\c{e}bicki, K., Michna, Z., and Peng, X. (2019).
\newblock Approximation of sojourn times of {G}aussian processes.
\newblock {\em Methodol. Comput. Appl. Probab.}, 21(4):1183--1213.

\bibitem[Dieker, 2005]{MR2111193}
Dieker, A.~B. (2005).
\newblock Extremes of {G}aussian processes over an infinite horizon.
\newblock {\em Stochastic Process. Appl.}, 115(2):207--248.

\bibitem[Dieker and Mikosch, 2015]{MR3351818}
Dieker, A.~B. and Mikosch, T. (2015).
\newblock Exact simulation of {B}rown-{R}esnick random fields at a finite
  number of locations.
\newblock {\em Extremes}, 18(2):301--314.

\bibitem[Dieker and Yakir, 2014]{MR3217455}
Dieker, A.~B. and Yakir, B. (2014).
\newblock On asymptotic constants in the theory of extremes for {G}aussian
  processes.
\newblock {\em Bernoulli}, 20(3):1600--1619.

\bibitem[Dieker, 2004]{dieker2004simulation}
Dieker, T. (2004).
\newblock {\em Simulation of fractional Brownian motion}.
\newblock PhD thesis, Masters Thesis, Department of Mathematical Sciences,
  University of Twente.

\bibitem[Guariglia, 2019]{MR3978987}
Guariglia, E. (2019).
\newblock Riemann zeta fractional derivative---functional equation and link
  with primes.
\newblock {\em Adv. Difference Equ.}, pages Paper No. 261, 15.

\bibitem[Ivanovs, 2018]{ivanovs2018zooming}
Ivanovs, J. (2018).
\newblock Zooming in on a {L}\'{e}vy process at its supremum.
\newblock {\em Ann. Appl. Probab.}, 28(2):912--940.

\bibitem[Jasnovidov and Shemendyuk, 2021]{thirdprojectParisian}
Jasnovidov, G. and Shemendyuk, A. (2021).
\newblock Parisian ruin for insurer and reinsurer under quota-share treaty.
\newblock {\em arXiv:2103.03213}.

\bibitem[Ji and Robert, 2018]{Lanpeng2BM}
Ji, L. and Robert, S. (2018).
\newblock Ruin problem of a two-dimensional fractional {B}rownian motion risk
  process.
\newblock {\em Stoch. Models}, 34(1):73--97.

\bibitem[Kabluchko and Wang, 2014]{MR3217426}
Kabluchko, Z. and Wang, Y. (2014).
\newblock Limiting distribution for the maximal standardized increment of a
  random walk.
\newblock {\em Stochastic Process. Appl.}, 124(9):2824--2867.

\bibitem[Kudryavtsev et~al., 2003]{Matan2003}
Kudryavtsev, L., Kutasov, A., Chehlov, V., and Shabunin, M. (2003).
\newblock {\em Collection of Problems in Mathematical Analysis. Part 2.
  Integrals and Series (in Russian)}.
\newblock M. Phizmatlit.

\bibitem[Pickands, 1969a]{MR250368}
Pickands, III, J. (1969a).
\newblock Asymptotic properties of the maximum in a stationary {G}aussian
  process.
\newblock {\em Trans. Amer. Math. Soc.}, 145:75--86.

\bibitem[Pickands, 1969b]{MR250367}
Pickands, III, J. (1969b).
\newblock Upcrossing probabilities for stationary {G}aussian processes.
\newblock {\em Trans. Amer. Math. Soc.}, 145:51--73.

\bibitem[Piterbarg, 1996]{Pit96}
Piterbarg, V.~I. (1996).
\newblock {\em Asymptotic methods in the theory of {G}aussian processes and
  fields}, volume 148 of {\em Translations of Mathematical Monographs}.
\newblock American Mathematical Society, Providence, RI.

\bibitem[Piterbarg, 2015]{20lectures}
Piterbarg, V.~I. (2015).
\newblock {\em Twenty Lectures About {G}aussian Processes}.
\newblock Atlantic Financial Press London New York.

\bibitem[Wood, 1992]{Polylogarithm_Properties}
Wood, D. (1992).
\newblock The computation of polylogarithms.
\newblock Technical Report 15-92*, University of Kent, Computing Laboratory,
  University of Kent, Canterbury, UK.

\end{thebibliography}
\end{document}